\documentclass[reqno,11pt]{amsart}
\usepackage{amssymb, amsmath}
\usepackage{amsthm, amsfonts,mathrsfs}
\usepackage[T1]{fontenc}
\usepackage{times}
\usepackage{mathrsfs}
\usepackage[pagewise]{lineno}
\usepackage{epsfig}
\usepackage{color}
\usepackage[all]{xypic}
\usepackage{amssymb}
\usepackage{cases}
\usepackage{subeqnarray}
\allowdisplaybreaks[4]
\newtheorem{theorem}{Theorem}[section]

\newtheorem{definition}{Definition}[section]
\newtheorem{lemma}{Lemma}[section]
\newtheorem{proposition}{Proposition}[section]
\newtheorem{remark}{Remark}[section]

\numberwithin{equation}{section}

\def\proofname{\noindent {\sl Proof.}}

\begin{document}
\title[Lifespan estimate for the semilinear
E-P-D-T equation]
{Lifespan estimate for the semilinear regular
Euler-Poisson-Darboux-Tricomi equation}

\author[Y.- Q. Li]{Yuequn Li}
\address{Yuequn Li
\newline
School of Mathematical Sciences, Nanjing Normal University Nanjing {210023}, China}
\email{yqli1214@163.com}

\author[F. Guo]{Fei Guo}
\address{Fei Guo \newline
School of Mathematical Sciences and Key Laboratory for NSLSCS,
Ministry of Education, Nanjing Normal University,
Nanjing 210023, China}
\email{guof@njnu.edu.cn}


\begin{abstract}
In this paper, we begin by establishing local well-posedness for the semilinear regular Euler-Poisson-Darboux-Tricomi equation. Subsequently, we derive a lifespan estimate with the Strauss index given by $p=p_{S}(n+\frac{\mu}{m+1}, m)$ for any $\delta>0$, where $\delta$ is a parameter to describe the interplay between damping and mass. This is achieved through the construction of a new test function derived from the Gaussian hypergeometric function and a second-order ordinary differential inequality, as proven by Zhou \cite{Zhou2014}. Additionally, we extend our analysis to prove a blow-up result with the index $p=\max\{p_{S}(n+\frac{\mu}{m+1}, m), p_{F}((m+1)n+\frac{\mu-1-\sqrt\delta}{2})\}$ by applying Kato$^{\prime}$s Lemma ( i.e., Lemma \ref{katolemma} ), specifically in the case of $\delta=1$.
\end{abstract}

\date{Sept. 10, 2024}

\maketitle



\vskip 0.2cm


\section{Introduction}\label{s1}

We consider the following  semilinear Euler-Poisson-Darboux-Tricomi equation (EPDT) with a power nonlinearity
\begin{equation}\label{eq}
  \left\{
\begin{aligned}
&\partial_t^2u-t^{2m}\Delta u+\frac{\mu}{t}\partial_tu+\frac{\nu^2}{t^2}u=|u|^p,\ &\quad& t\geq1,\\
&u(1,x)=\varepsilon u_0(x), \ \partial_tu(1,x)=\varepsilon u_1(x),               &\quad&x\in\mathbb{R}^n,
\end{aligned}
 \right.
  \end{equation}
where $m\geq0$, $n\geq1$ denotes the space dimension and $\varepsilon>0$ is a sufficiently small number describing  the smallness of the initial value.  Palmieri \cite{Pa2021} first named (\ref{eq}) the "EPDT" while investigating  its critical exponent problem, one can see \cite{Ab2021}, \cite{ShSi2014} and their references to obtain more results on the generalized Euler-Poisson-Darboux equation. (\ref{eq}) is known as scale-invariant since the corresponding linear equation is invariant under the scaling
\begin{equation}\label{scale}
v(t,x)=u(\lambda^{\frac{1}{m+1}}t, \lambda x).
\end{equation}

(\ref{eq}) becomes the generalized Euler-Poisson-Darboux equation  when $m=\nu=0$, which  is commonly encountered in the fields of gas dynamics, fluid mechanics, elasticity, plasticity, and various physical and material science problems.  If the initial values are given at time zero, this problem is referred to as the singular problem, and if the initial values are given at a non-zero time, it is called the regular problem. D$^\prime$Abbicco, Lucente and Reissig \cite{AbLuRe2015}  determined the critical exponent for the regular problem and  proved  blow-up results in the case of $n\geq2$ with $\mu=2$, and the global well-posedness was only proved in low-dimensional cases ($n=2, 3$). Subsequently, they \cite{AbLu2015(2)} demonstrated the existence of global solutions for the regular problem with spherically symmetric initial data. However, concerning the singular problem, to our best knowledge, there are currently no satisfactory results available and some  partial results can be referred to \cite{Ab2021, Ab2020}.

When $\mu=\nu^2=0$,  (\ref{eq})  reduces to the  semilinear Tricomi equation
 \begin{equation}\label{eq2}
\begin{aligned}
\partial_t^2u-t^{2m}\Delta u=|u|^p,
\end{aligned}
  \end{equation}
and it is  often  used in gas dynamics to describe the conversion between ultrasound and infrasound, one can refer \cite{Frankel1945} for more details. At the early stage,  mathematicians attempted to establish the expression of explicit solutions for the linear Tricomi equation.  Barros-Neto and Gelfand \cite{BaGe1999, BaGe2002, BaGe2005} found  the fundamental solution for 1-D problem $yu_{xx}+u_{yy}=0$ and then Yagdjian \cite{Yagdjian2004}  established the explicit fundamental solution for the generalized Gellerstedt operator $\partial_t^2-t^{2m}\Delta$ in  $\mathbb{R}^{n+1}$.
In recent years, the large time behavior of nonlinear Tricomi equations  has attracted much attention.  Many mathematicians dedicated to  the critical exponent problem for (\ref{eq2}), namely, it is expected that there exists an index  $p_c(n,m)$ such that if $p>p_c(n,m)$, (\ref{eq2}) has small data global solution and when $1<p\leq p_c(n,m)$, the solution to (\ref{eq2}) will blow up even for small initial values. 
Note that when $m=0$, (\ref{eq2}) becomes the classical semilinear wave equation
\begin{equation}\label{eq3}
\partial_t^2u-\Delta u=|u|^p
\end{equation}
with the critical exponent $p_S(n)$,  which is the positive root of the  quadratic equation
 \begin{equation}\label{eeq1}
(n-1)p^2-(n+1)p-2=0\ \text{for}\ n\geq2,
\end{equation}
and $p_S(1)=+\infty$,
one can refer \cite{GeLiSo1997, Glassey1981, John1979, Strauss1981, Ta2001, Zh2006, Zhou2007, Zhou2014} for more details about Strauss conjecture. In the past few years, He, Witt and Yin \cite{He2017II, He2017I, He2020} systematically studied (\ref{eq2}) and determined the critical exponent $p_S(n,m)$, more precisely,  $p_S(n,m)$ is the positive root of the quadratic equation
\begin{equation}\label{eeq2}
\bigl((m+1)n-1\bigr)p^2-\bigl((m+1)(n-2)+3\bigr)p-2(m+1)=0.
\end{equation}
It should be pointed out that the critical exponent $p_S(n,m)$ extends $p_S(n)$ because (\ref{eeq2}) reduces to  (\ref{eeq1}) when $m=0$, that is $p_S(n)=p_S(n,0)$.  Subsequently, Lin and Tu \cite{LinTu2019} obtained the upper bound estimate of lifespan by the test function method when $1<p\leq p_S(n,m)$, and  Sun \cite{Sun2021} gave the lower bound estimate based on the Strichartz estimate established in \cite{He2020}  when $1<p< p_S(n,m)$.

When $m=0$, $(\ref{eq})$ becomes the semilinear wave equation with scale invariant damping and mass
\begin{equation}\label{eq4}
  \left\{
\begin{aligned}
&\partial_t^2u-\Delta u+\frac{\mu}{t}\partial_tu+\frac{\nu^2}{t^2}u=|u|^p, &\quad&t\geq1,\\
&u(1,x)=\varepsilon u_0(x), \ \partial_tu(1,x)=\varepsilon u_1(x),               &\quad&x\in\mathbb{R}^n.
\end{aligned}
 \right.
  \end{equation}
It has been found that the scale of the non-negative number $\delta=(\mu-1)^2-4\nu^2$   reflects an important relationship between the damping term $\frac{\mu}{t}\partial_tu$ and the mass term $\frac{\nu^2}{t^2}u$. Specifically speaking, in \cite{PaRe2019}, it has been proved that the solution to (\ref{eq4}) blows up when $1<p\leq\max\{p_S(n+\mu), p_{F}(n+\frac{\mu-1-\sqrt\delta}{2})$ for $\delta\in(0,1]$  in high dimension cases ($n\geq2$), where  $p_{F}(n)=1+\frac{2}{n}$ denotes the Fujita index appearing in the semilinear heat equation (c.f. \cite{Fu1966, Ha1973}).  For large $\delta$ ($\delta\geq(n+1)^2$), (\ref{eq4}) can be called  ``parabolic-like''  since it has been proved that $p_{F}(n+\frac{\mu-1-\sqrt\delta}{2})$ is the critical exponent in \cite{PaRe2017, Pa2018, PaRe2018}.  Subsequently,  it is shown in \cite{Tu2019} that the solution blows up when $1<p< p_S(n+\mu)$ for all $\delta>0$ and $p=p_S(n+\mu)$ for $\delta<n^2$.

As for (\ref{eq}), Palmieri \cite{Pa2021} proved that for all $\delta\geq0$, the solution blows up when $1<p<\max\{p_S(n+\frac{\mu}{m+1},m), p_{F}((m+1)n+\frac{\mu-1-\sqrt\delta}{2})\}$ and established the upper bound estimate of the lifespan utilizing the iterative argument, where $p_S(n+\frac{\mu}{m+1},m)$ is the positive root to the following quadratic equation
\begin{equation}\label{Lambdaa}
\bigl((m+1)n-1+\mu\bigr)p^2-\bigl((m+1)(n-2)+3+\mu\bigr)p-2(m+1)=0.
\end{equation}
Note that $p_S(n+\mu,0)=p_S(n+\mu)$, so we can make a reasonable conjecture that,   for different values of $\delta$, $p_S(n+\frac{\mu}{m+1},\ m)$, $p_{F}((m+1)n+\frac{\mu-1-\sqrt\delta}{2})$ or  the larger of the two indices represents the critical exponent of (\ref{eq}). The primary objective in this paper is to demonstrate that for $\delta < (m+1)^2n^2$, the solution to (\ref{eq}) blows up in finite time and give the upper bound lifespan estimate when $p=p_S(n+\frac{\mu}{m+1},m)$. Due to the appearance of $t^m$ in the Gellerstedt operator, the test function that works for the wave operator is no longer applicable, this compels us to search for a new one that depends on $m$. Inspired by the treatment of $p=p_S(n,0)$ in \cite{Tu2019}, we construct a new test function using Gaussian hypergeometric functions, and when $m=0$, this function is the same as the one used in \cite{Tu2019}. Additionally, when $\delta=1$,   (\ref{eq}) can be transformed into a Tricomi-type equation without damping and mass terms, and then the second objective of this paper is to  give a blow up result for $p=\max\{p_{S}(n+\frac{\mu}{m+1}, m), p_{F}((m+1)n+\frac{\mu-1-\sqrt\delta}{2})\}$ under the condition $\delta=1$. It is a pity that  when $\delta \neq 1$, we have not found a suitable method to handle the case of  $p=p_{F}((m+1)n+\frac{\mu-1-\sqrt\delta}{2})$, however, in our next paper, we demonstrate that $p_{F}((m+1)n+\frac{\mu-1-\sqrt\delta}{2})$ is the critical exponent for large $\delta$.

 This paper is organized as follows. In Section 2, we provide the definition of the energy solution to (\ref{eq}) and present the main results.  Section 3 is devoted to prove the local well-posedness of  (\ref{eq}) by deriving $(L^1\cap L^2)-L^2$ estimates to the solutions of the corresponding linear equation.  In Section 4,  we give the new test function, and consequently, obtain the upper bound estimate of the lifespan with the Strauss index given by $p=p_{S}(n+\frac{\mu}{m+1}, m)$.   A blow up result is derived in Section 5 in the case of $\delta=1$ with the index $p=\max\{p_{S}(n+\frac{\mu}{m+1}, m), p_{F}((m+1)n+\frac{\mu-1-\sqrt\delta}{2})\}$.

\section{Main results}

Before stating our results, we give the definition of energy solution for (\ref{eq}). Let us introduce the space
\begin{equation}\label{norm1}
D^\sigma(\mathbb{R}^n)=\left\{
\begin{aligned}
&\bigl(L^1(\mathbb{R}^n)\cap H^\sigma(\mathbb{R}^n)\bigr)\times\bigl(L^1(\mathbb{R}^n)\cap H^{\sigma-1}(\mathbb{R}^n)\bigr),&\quad&\sigma>1,\\
&\bigl(L^1(\mathbb{R}^n)\cap H^\sigma(\mathbb{R}^n)\bigr)\times\bigl(L^1(\mathbb{R}^n)\cap L^2(\mathbb{R}^n)\bigr), &\quad&\sigma\in [0,1]
\end{aligned}
 \right.
\end{equation}
with the norm
\begin{equation}\label{norm2}
\Vert(u_0,u_1)\Vert_{D^\sigma}=\left\{
\begin{aligned}
&\Vert u_0\Vert_{L^1}+\Vert u_0\Vert_{H^\sigma}+\Vert u_1\Vert_{L^1}+\Vert u_1\Vert_{H^{\sigma-1}},&\quad&\sigma>1,\\
&\Vert u_0\Vert_{L^1}+\Vert u_0\Vert_{H^\sigma}+\Vert u_1\Vert_{L^1}+\Vert u_1\Vert_{L^{2}}, &\quad&\sigma\in [0,1],
\end{aligned}
 \right.
\end{equation}
and $D^1:=D.$

\begin{definition}\label{def2}
Let $(u_0,u_1)\in D$ be compactly supported in $\mathbb{R}^n$ and $u\in L_{loc}^p\bigl([1,T)\times\mathbb{R}^n \bigr)$.  We say that  $u$ is an energy solution of (\ref{eq}) if
\begin{equation*}
u\in C\bigl([1,T); H^1(\mathbb{R}^n)\bigr) \cap C^1\bigl([1,T); L^2(\mathbb{R}^n)\bigr)
\end{equation*}  satisfies for any $\varPhi(t,x)\in C_0^{\infty}\bigl([1,T)\times \mathbb{R}^n\bigr)$ and $t\in[1,T)$,
\begin{equation}\label{def3}
\resizebox{   0.918\hsize}{!}{$  \begin{split}
&\int_1^t\int_{\mathbb{R}^n}u(s,x)\Bigl(\partial_s^2\varPhi(s,x)-s^{2m}\Delta\varPhi(s,x)-\partial_t\bigl(\frac{\mu}{s}\varPhi(s,x)\bigr)+\frac{\nu^2}{s^2}\varPhi(s,x)\Bigr)dxds\\
&\quad+\int_{\mathbb{R}^n}\bigl(\partial_tu(t,x)\varPhi(t,x)-u(t,x)\partial_t\varPhi(t,x)+\frac{\mu}{t}u(t,x)\varPhi(t,x)\bigr)dx\\
&=\int_1^t\int_{\mathbb{R}^n}\vert u(s,x)\vert^p\varPhi(s,x)dxds\\
&\quad+\varepsilon\int_{\mathbb{R}^n}\Bigl(-u_0(x)\partial_t\varPhi(1,x)+\bigl(\mu u_0(x)+u_1(x)\bigr)\varPhi(1,x)\Bigr)dx.
\end{split}     $}
\end{equation}
\end{definition}

To begin with, we establish the local well-posedness of  (\ref{eq}).
\begin{proposition}
Let $(u_0,u_1)\in D$ compactly supported in $B_M(0)=\{x: \vert x\vert\leq M\}$ with some $M>0$. Suppose
\begin{equation}\label{GNNN}
1<p\leq\frac{n}{n-2},\ n\geq3; \ 1<p<\infty,\ n=1, 2
\end{equation}
and $\mu$, $\nu\geq 0$ such that
\begin{equation}\label{delta}
\delta:=(\mu-1)^2-4\nu^2>0,
\end{equation}
then, there exists a $T>1$ and a unique solution $u\in C\bigl([1,T); H^1(\mathbb{R}^n)\bigr)\cap C^1\bigl([1,T); L^2(\mathbb{R}^n)\bigr)$ to (\ref{eq}).
\end{proposition}


The lifespan estimate of the solutions to (\ref{eq}) is as follows.

\begin{theorem}\label{theorem2}
Let $m$, $\mu$, $\nu$ $\geq0$, $n\geq1$ and $0<\delta<(m+1)^2n^2$. Assume that the initial data $(u_0,u_1)\in D$ be nonnegative, not identically zero and $supp(u_0,u_1)\subset\{x:\vert x\vert\leq M\}$ for some $0<M<\frac{1}{m+1}$. Consider $p=p_S(n+\frac{\mu}{m+1},m)$, we suppose $p>\frac{2(m+1)}{(m+1)n-\sqrt\delta}$, furthermore, when $n=1$, we assume $\mu\geq m$, then  there exists a $\varepsilon_0=\varepsilon_0(u_0, u_1, m, n, p, M, \mu, \nu)>0$ such that for any $0<\varepsilon\leq\varepsilon_0$, the solution u to (\ref{eq}) blows up in finite time and we have the upper estimate of the lifespan $T(\varepsilon)$
\begin{equation*}
T(\varepsilon)\leq e^{C\varepsilon^{-p(p-1)}}.
\end{equation*}
\end{theorem}

\begin{remark}
We assess the reasonability for the assumption $p>\frac{2(m+1)}{(m+1)n-\sqrt\delta}$. Since $p=p_S(n+\frac{\mu}{m+1},m)$,  the assumption is  actually a restriction on $\mu$, $\nu$.  We can easily show that  $n=3$, $\mu=3$ and $\nu=\frac{1}{2}$  are suitable for $0<3=\delta<9(m+1)^2$ and this assumption,
furthermore, we can check that (\ref{GNNN}) holds in this case.
\end{remark}


\begin{theorem}\label{theorem3}
Assume that $m$, $\mu$, $\nu$ $\geq0$, $n\geq1$ and $\delta=1$,   the initial data $(u_0,u_1)\in D$ be nonnegative, not identically zero and compactly supported in $\{x:\vert x\vert\leq M\}$ for some $M>0$.   Then the solution to (\ref{eq}) blows up in finite time when $p=\max \{p_S(n+\frac{\mu}{m+1},m), p_{F}((m+1)n+\frac{\mu-1-\sqrt\delta}{2})\}$.
\end{theorem}

\begin{remark}
After a complex algebraic calculation, we discover that $p_S(n+\frac{\mu}{m+1},m) > p_{F}((m+1)n+\frac{\mu-1-\sqrt\delta}{2})$ holds when $n\geq3$ in the case of $\delta=1$ for any $m\geq0$, then in view of Theorem \ref{theorem2},  we only need to deal with the low dimensions $n=1,2$,  and the details can be found in Section 5.
\end{remark}

\section{The Proof framework of Proposition 2.1}

Through two changes of variables, we can transform (\ref{eq}) into a damped wave equation without the mass term. By following the approach in  \cite{PaRe2018} and  \cite{Wi2004}, we can establish ($L^1\cap L^2)-L^2$ estimates for the corresponding linear equation of (\ref{eq}). Subsequently, the existence of local solutions can be obtained by the contraction mapping principle.

Consider the corresponding linear equation of (\ref{eq})
\begin{equation}\label{eqq}
  \left\{
\begin{aligned}
&\partial_t^2u-t^{2m}\Delta u+\frac{\mu}{t}\partial_tu+\frac{\nu^2}{t^2}u=0,\ &\quad&t\geq\tau\geq1,\\
&u(\tau,x)=u_0(x), \ \partial_tu(\tau,x)= u_1(x),               &\quad&x\in\mathbb{R}^n,
\end{aligned}
 \right.
  \end{equation}
let
\begin{equation}\label{change1}
u(t,x)=t^{\frac{\sqrt{\delta}-\mu+1}{2}}w(t,x),
\end{equation}
where $\delta$ is defined by (\ref{delta}),   (\ref{eqq}) is transformed into
\begin{equation}\label{eqq2}
 \left\{
\begin{aligned}
&\partial_t^2w-t^{2m}\Delta w+\bigl(1+\sqrt\delta\bigr)\frac{\partial_tw}{t}=0,\  &\quad&t\geq\tau\geq1,\\
&w(\tau,x)= w_0(x), \ \partial_tw(\tau,x)= w_1(x),      &\quad&x\in\mathbb{R}^n,
\end{aligned}
 \right.
  \end{equation}
where
\begin{equation*}
\begin{aligned}
w_0(x)&=\tau^{-\frac{\sqrt\delta-\mu+1}{2}}u_0(x), \\
w_1(x)&=\tau^{-\frac{\sqrt\delta-\mu+1}{2}}u_1(x)-\frac{\sqrt\delta-\mu+1}{2}\tau^{-\frac{\sqrt\delta-\mu+3}{2}}u_0(x).
\end{aligned}
\end{equation*}
Set
\begin{equation}\label{change2}
s=\frac{t^{m+1}}{m+1},\ \eta=\frac{\tau^{m+1}}{m+1}\ \text{and} \ v(s,x)=w\Bigl(\bigl((m+1)s\bigr)^{\frac{1}{m+1}}, x\Bigr),
\end{equation}
(\ref{eqq2}) can be transformed into
\begin{equation}\label{eqq3}
 \left\{
\begin{aligned}
&\partial_s^2v-\Delta v+\bigl(1+\frac{\sqrt\delta}{m+1}\bigr)\frac{\partial_sv}{s}=0,&\quad&s\geq\eta, \\
&v(\eta,x)= v_0(x), \ \partial_sv(\eta,x)= v_1(x), &\quad&x\in\mathbb{R}^n,
\end{aligned}
 \right.
  \end{equation}
where
\begin{equation*}
\begin{aligned}
v_0(x)&=w_0(x)=\bigl((m+1)\eta\bigr)^{-\frac{\sqrt\delta-\mu+1}{2(m+1)}}u_0(x), \\
v_1(x)&=\bigl((m+1)\eta\bigr)^{-\frac{m}{m+1}}w_1(x)\\
      &=\bigl((m+1)\eta\bigr)^{-\frac{m}{m+1}-\frac{\sqrt\delta-\mu+1}{2(m+1)}}u_1(x)-\tiny{\frac{\sqrt\delta-\mu+1}{2}}\bigl((m+1)\eta\bigr)^{-\frac{m}{m+1}-\frac{\sqrt\delta-\mu+3}{2(m+1)}}u_0(x).
\end{aligned}
\end{equation*}

We will use the ideas outlined in \cite{PaRe2018} to establish ($L^1\cap L^2)-L^2$ estimates for (\ref{eqq3}), subsequently, the  ($L^1\cap L^2)-L^2$ estimates for (\ref{eqq}) can be obtained through (\ref{change1}) and (\ref{change2}). Here, we only list the results.

\begin{lemma}\label{estimate1}
Assume that $u_0=0$, $u_1\in L^{1}\cap H^{k}$ for some $k>0$ and $\delta>0$. Then for any $k\in [0,1]$, the solution $u$ to (\ref{eqq}) satisfies the decay estimate
\[
\Vert u(t,\cdot)\Vert_{\dot{H}^k}\leq C_m \bigl(\Vert u_1\Vert_{L^1}+\tau^{(m+1)\frac{n}{2}}\Vert u_1\Vert_{L^2}\bigr)\quad\quad\quad\quad\quad\quad\quad\quad\quad
\]
\begin{subequations}\label{linear estimates1}
\begin{numcases}{\times}
t^{-\frac{\mu+m}{2}}\tau^{-(m+1)(k+\frac{n}{2})+\frac{\mu+m}{2}+1},\ \ \ \quad \text{if} \quad  k>\frac{\sqrt\delta}{2(m+1)}+\frac{1}{2}-\frac{n}{2},\label{3.6a}\\
t^{-\frac{\mu+m}{2}}\tau^{-\frac{\sqrt\delta-\mu-1}{2}}(1+\log \frac{t}{\tau})^{\frac{1}{2}}, \ \quad \text{if}\quad  k=\frac{\sqrt\delta}{2(m+1)}+\frac{1}{2}-\frac{n}{2},\label{3.6b}\\
t^{-(m+1)(k+\frac{n}{2})+\frac{\sqrt\delta-\mu+1}{2}}\tau^{-\frac{\sqrt\delta-\mu-1}{2}},\ \text{if} \quad k<\frac{\sqrt\delta}{2(m+1)}+\frac{1}{2}-\frac{n}{2}.\label{3.6c}
\end{numcases}
\end{subequations}
Moreover, for any $k\geq1$,
\[
\Vert \partial_tu(t,\cdot)\Vert_{\dot{H}^{k-1}}\leq C_m \bigl(\Vert u_1\Vert_{L^1}+\tau^{(m+1)(\frac{n}{2}+k-1)}\Vert u_1\Vert_{\dot{H}^{k-1}}+\tau^{(m+1)(\frac{n}{2}+[k-2]_+)}\Vert u_1\Vert_{\dot{H}^{[k-2]_+}}\bigr)\quad\quad\quad\quad\quad\quad\quad\quad\quad
\]
\begin{subequations}\label{linear estimates1}
\begin{numcases}{\quad\times}
t^{m-\frac{\mu+m}{2}}\tau^{-(m+1)(k+\frac{n}{2})+\frac{\mu+m}{2}+1},\quad \quad \text{if} \quad  k>\frac{\sqrt\delta}{2(m+1)}+\frac{1}{2}-\frac{n}{2},\label{3.7a}\\
t^{m-\frac{\mu+m}{2}}\tau^{-\frac{\sqrt\delta-\mu-1}{2}}(1+\log \frac{t}{\tau})^{\frac{1}{2}}, \ \  \quad \text{if}\quad  k=\frac{\sqrt\delta}{2(m+1)}+\frac{1}{2}-\frac{n}{2},\label{3.7b}\\
t^{m-(m+1)(k+\frac{n}{2})+\frac{\sqrt\delta-\mu+1}{2}}\tau^{-\frac{\sqrt\delta-\mu-1}{2}},\ \ \text{if} \quad k<\frac{\sqrt\delta}{2(m+1)}+\frac{1}{2}-\frac{n}{2},\label{3.7c}
\end{numcases}
\end{subequations}
where $[x]_+=\max\{x,0\}$.
\end{lemma}

\begin{lemma}\label{estimate2}
Assume that $(u_0,u_1)\in D^k$ for some $k>0$ and $\delta>0$. Then for any $k\in [0,1]$, the solution $u$ to (\ref{eqq}) with $\tau=1$ satisfies  the decay estimates
\[
\Vert u(t,\cdot)\Vert_{\dot{H}^k}\leq C_m \Vert (u_0,u_1)\Vert_{D^k}\quad\quad\quad\quad\quad\quad\quad\quad\quad
\]
\begin{subequations}\label{linear estimates1}
\begin{numcases}{\quad\quad\quad\quad\times}
t^{-\frac{\mu+m}{2}},\quad \quad \quad \quad \quad \quad \quad \text{if} \quad k>\frac{\sqrt\delta}{2(m+1)}+\frac{1}{2}-\frac{n}{2},\label{3.8a}\\
t^{-\frac{\mu+m}{2}}(1+\log t)^{\frac{1}{2}}, \quad \quad  \text{if}\quad  k=\frac{\sqrt\delta}{2(m+1)}+\frac{1}{2}-\frac{n}{2},\label{3.8b}\\
t^{-(m+1)(k+\frac{n}{2})+\frac{\sqrt\delta-\mu+1}{2}}, \ \ \text{if} \quad  k<\frac{\sqrt\delta}{2(m+1)}+\frac{1}{2}-\frac{n}{2},\label{3.8c}
\end{numcases}
\end{subequations}
and for any $k\geq1$,
\[
\Vert \partial_tu(t,\cdot)\Vert_{\dot{H}^{k-1}}\leq C_m \Vert (u_0,u_1)\Vert_{D^k}\quad\quad\quad\quad\quad\quad\quad\quad\quad
\]
\begin{subequations}\label{linear estimates1}
\begin{numcases}{\quad\quad\quad\quad\times}
t^{m-\frac{\mu+m}{2}}, \quad\quad\quad\quad\quad\quad\ \ \text{if} \quad k>\frac{\sqrt\delta}{2(m+1)}+\frac{1}{2}-\frac{n}{2},\label{3.9a}\\
t^{m-\frac{\mu+m}{2}}(1+\log t)^{\frac{1}{2}}, \ \ \quad \text{if}\quad  k=\frac{\sqrt\delta}{2(m+1)}+\frac{1}{2}-\frac{n}{2},\label{3.9b}\\
t^{m-(m+1)(k+\frac{n}{2})+\frac{\sqrt\delta-\mu+1}{2}},\ \text{if} \quad  k<\frac{\sqrt\delta}{2(m+1)}+\frac{1}{2}-\frac{n}{2}.\label{3.9c}
\end{numcases}
\end{subequations}
\end{lemma}

Denote by $E_0(t,1,x)$ and $E_1(t,1,x)$ the  fundamental solution to (\ref{eqq}) with initial data $(u_0,u_1)=(\delta_0,0)$ and $(0,\delta_0)$, respectively, where $\delta_0$ is the Dirac function. By Duhamel's principle, $\int_1^tE_1(t,\tau,x)\ast_{(x)}F(\tau,x)d\tau$
is the solution to the inhomogeneous linear Cauchy problem
\begin{equation*}
  \left\{
\begin{aligned}
&\partial_t^2u-t^{2m}\Delta u+\frac{\mu}{t}\partial_tu+\frac{\nu^2}{t^2}u=F(t,x),\ &\quad&t\geq\tau\geq1,\\
&t=\tau:\ u=0, \ \partial_tu=0,&\quad& x\in \mathbb{R}^n,
\end{aligned}
 \right.
  \end{equation*}
where $\ast_{(x)}$ denotes the convolution with respect to $x$.
Introduce the Banach space
\begin{equation}\label{space def}
\begin{split}
X(T):=\bigg\{&u\in C([1,T]; H^1)\cap C^1([1,T]; L^2) \\
             &\text{such that supp} u(t,\cdot) \subset B_{\phi(t)-\phi(1)+M}\ \text{for all}~t \in [1,T]\bigg\}
\end{split}
\end{equation}
equipped with the norm  $\Vert u\Vert_{X(T)}=\mathop{\max}\limits_{t\in[1,T]} M[u](t)$, where $M[u](t):=\Vert u(t,\cdot)\Vert_{L^2}+\Vert \partial_tu(t,\cdot)\Vert_{L^2}+\Vert \nabla u(t,\cdot)\Vert_{L^2}$.
Let
\begin{equation}\label{X(T,K) def}
  \begin{split}
X(T,K):=\left\{u\in X(T): \|u\|_{X(T)}\leqslant K\right\}
\end{split}
\end{equation}
for some $T>1,K>0$. We define the operator
\begin{equation*}
\mathcal{N}: u\in X(T)\longrightarrow \mathcal{N}(u)=u^l+u^n,
\end{equation*}
where  $u^l=E_0(t,1,x)\ast_{(x)} u_0(x)+E_1(t,1,x)\ast_{(x)} u_1(x)$ is the solution to (\ref{eqq}) with $\tau=1$,  and $u^n=\int_1^tE_1(t,\tau,x)\ast_{(x)}\vert u(\tau,x)\vert^pd\tau$, i.e., $\mathcal{N}(u)$ is the solution to
\begin{equation*}
  \left\{
\begin{aligned}
&\partial_t^2u-t^{2m}\Delta u+\frac{\mu}{t}\partial_tu+\frac{\nu^2}{t^2}u=\vert u\vert^p,&\quad&t\geq1,\\
&u(1,x)=u_0(x), \ \partial_tu(1,x)= u_1(x),              &\quad&x\in\mathbb{R}^n.
\end{aligned}
 \right.
  \end{equation*}
Once we have proved  that for any $u,\tilde{u}\in X(T,K)$,
\begin{align}
&\Vert\mathcal{N}(u)\Vert_{X(T)}\leq C_{1T}\Vert (u_0,u_1)\Vert_{D}+C_{2T}\Vert u\Vert_{X(T)}^{p},\label{kkey1}\\
&\Vert\mathcal{N}(u)-\mathcal{N}(\tilde{u})\Vert_{X(T)}\leq C_{3T}\Vert u-\tilde{u}\Vert_{X(T)}\bigl(\Vert u\Vert_{X(T)}^{p-1}+\Vert \tilde{u}\Vert_{X(T)}^{p-1}\bigr),\label{kkey2}
\end{align}
where $C_{1T}$ is bounded,  $C_{2T}, C_{3T}\rightarrow0$ as $T\rightarrow 1^{+}$, we can obtain the local solution to (\ref{eq}).  Indeed, for sufficiently large  $K$, we
can choose  $T$ sufficiently close to $1$, ensuring that  $\mathcal{N}$ maps $X(T,K)$ into itself and $\mathcal{N}$ is a contraction mapping, therefore, we can get the local existence and uniqueness of the solution to (\ref{eq}) in $X(T)$ by the contraction mapping principle.

Our aim now is to prove (\ref{kkey1}) and (\ref{kkey2}) based on Lemmas \ref{estimate1}-\ref{estimate2}.  The Gagliardo-Nirenberg inequality will be used, one can refer to \cite{Fr1976}.

\begin{lemma}(Gagliardo$-$Nirenberg Inequality)\label{GN}
Let $1< p,q,r<\infty$ and $\theta\in[0,1]$ satisfy
\begin{equation*}
\frac{1}{p}=\theta(\frac{1}{r}-\frac{1}{n})+(1-\theta)\frac{1}{q}.
\end{equation*}
 Then there exists a constant $C=C(p,q,r,n)>0$ such that for any $h\in C_0^1(\mathbb{R}^n)$,
\begin{equation}\label{gnineq}
\Vert h\Vert_{L^p}\leq C\Vert h\Vert_{L^q}^{1-\theta}\Vert \nabla h\Vert_{L^r}^{\theta}.
\end{equation}
\end{lemma}

We proceed with the proof by considering five cases.\\
Case 1: $\frac{\sqrt\delta}{2(m+1)}+\frac{1}{2}-\frac{n}{2}>1$, i.e., $\delta>(m+1)^2(n+1)^2$.

Taking $k=0,1$ in (\ref{3.8c}) and $k=1$ in (\ref{3.9c}) respectively, we get
\begin{equation*}
\begin{aligned}
&\Vert u^l(t,\cdot)\Vert_{L^2}\leq C_mt^{-(m+1)\frac{n}{2}+\frac{\sqrt\delta-\mu+1}{2}}\Vert (u_0,u_1)\Vert_{D^0}\leq C_mt^{\frac{\sqrt\delta+1}{2}}\Vert (u_0,u_1)\Vert_{D},\\
&\Vert\nabla u(t,\cdot)\Vert_{L^2}\leq C_mt^{-(m+1)(1+\frac{n}{2})+\frac{\sqrt\delta-\mu+1}{2}}\Vert (u_0,u_1)\Vert_{D}\leq C_mt^{\frac{\sqrt\delta+1}{2}}\Vert (u_0,u_1)\Vert_{D},\\
&\Vert\partial_tu(t,\cdot)\Vert_{L^2}\leq C_mt^{m-(m+1)(1+\frac{n}{2})+\frac{\sqrt\delta-\mu+1}{2}}\Vert (u_0,u_1)\Vert_{D}\leq C_mt^{\frac{\sqrt\delta+1}{2}}\Vert (u_0,u_1)\Vert_{D},
\end{aligned}
\end{equation*}
which yields
\begin{equation}\label{ul}
M[u^l](t)\leq C_mt^{\frac{\sqrt\delta+1}{2}}\Vert (u_0,u_1)\Vert_{D}.
\end{equation}
By Lemma \ref{GN}, H$\ddot{\text{o}}$lder$^{\prime}$s inequality and the definition of $X(T)$, we have
\begin{equation}\label{gn1}
\begin{aligned}
\Vert \vert u(\tau,\cdot)\vert^p\Vert_{L^2}&=\Vert u(\tau,\cdot)\Vert_{L^{2p}}^p\\
 &\leq C\Vert u(\tau,\cdot)\Vert_{L^2}^{\theta(2p)p}\Vert \nabla u(\tau,\cdot)\Vert_{L^2}^{(1-\theta(2p))p}\leq CM[u]^p(\tau)
\end{aligned}
\end{equation}
and
\begin{equation}\label{gn2}
\begin{aligned}
\Vert \vert &u(\tau,\cdot)\vert^p\Vert_{L^1}=\Vert u(\tau,\cdot)\Vert_{L^{p}}^p\leq C\Vert u(\tau,\cdot)\Vert_{L^{2p}}^{p}(\phi(\tau)-\phi(1)+M)^{\frac{n}{2}}\\
&\leq C\tau^{\frac{(m+1)n}{2}}\Vert u(\tau,\cdot)\Vert_{L^2}^{\theta(2p)p}\Vert \nabla u(\tau,\cdot)\Vert_{L^2}^{(1-\theta(2p))p}\leq C\tau^{\frac{(m+1)n}{2}}M[u]^p(\tau),
\end{aligned}
\end{equation}
where
\begin{equation}\label{theta2p}
\theta(2p)=\frac{n(p-1)}{2p}
\end{equation}
 and the condition (\ref{GNNN}) ensures that $\theta(2p)\in[0,1].$
Then by taking $k=0$ in (\ref{3.6c}),  using Duhamel's principle and (\ref{gn1})-(\ref{gn2}), we can estimate $\Vert u^n(t,\cdot)\Vert_{L^2}$ as
\begin{equation}\label{unl2}
\begin{aligned}
&\Vert u^n(t,\cdot)\Vert_{L^2}\leq C_m t^{-(m+1)\frac{n}{2}+\frac{\sqrt\delta-\mu+1}{2}}\\
&\quad\quad\quad\quad\times\int_1^t\tau^{-\frac{\sqrt\delta-\mu-1}{2}}\big(\Vert \vert u(\tau,\cdot)\vert^p\Vert_{L^1}+\tau^{\frac{(m+1)n}{2}}\Vert \vert u(\tau,\cdot)\vert^p\Vert_{L^2}\big)d\tau\\
&\leq C_mt^{-(m+1)\frac{n}{2}+\frac{\sqrt\delta-\mu+1}{2}}\int_1^t\tau^{\frac{(m+1)n}{2}-\frac{\sqrt\delta-\mu-1}{2}}M[u]^p(\tau)d\tau\\
&\leq C_mt^{-(m+1)\frac{n}{2}+\frac{\sqrt\delta-\mu+1}{2}+\frac{(m+1)n+\mu+1}{2}}\int_1^tM[u]^p(\tau)d\tau\\
&\leq C_mt^{\frac{\sqrt\delta}{2}+1}\int_1^tM[u]^p(\tau)d\tau.
\end{aligned}
\end{equation}
The estimates of $\Vert \nabla u(t,\cdot)\Vert_{L^2}$ and $\Vert \partial_tu(t,\cdot)\Vert_{L^2}$ can be derived by taking $k=1$ in (\ref{3.6c}) and (\ref{3.7c}) and we list the results as
\begin{align}
&\Vert \nabla u^n(t,\cdot)\Vert_{L^2}\leq C_mt^{\frac{\sqrt\delta}{2}+1}\int_1^tM[u]^p(\tau)d\tau,\label{nablaul2}\\
&\Vert \partial_tu^n(t,\cdot)\Vert_{L^2}\leq C_mt^{\frac{\sqrt\delta}{2}}\int_1^tM[u]^p(\tau)d\tau.\label{partialtul2}
\end{align}
Combining (\ref{unl2})-(\ref{partialtul2}) yields
\begin{equation}\label{un}
M[u^n](t)\leq C_mt^{\frac{\sqrt\delta}{2}+1}\int_1^tM[u]^p(\tau)d\tau.
\end{equation}
It follows from (\ref{ul}) and (\ref{un}) that $M[\mathcal{N}u](t)\leq C_mt^{\frac{\sqrt\delta}{2}+1}\big(\Vert(u_0,u_1)\Vert_{D}+\int_1^tM[u]^p(\tau)d\tau\big)$, which indicates the
 validness of (\ref{kkey1}), where $C_{1T}=C_mT^{\frac{\sqrt\delta}{2}+1}$ and $C_{2T}=C_mT^{\frac{\sqrt\delta}{2}+1}(T-1)$.

It remains to prove (\ref{kkey2}). Again using Lemma \ref{GN} and  H$\ddot{\text{o}}$lder$^{\prime}$s inequality, we get
\begin{equation}\label{uuuu1}
\begin{aligned}
\Vert &\vert u(\tau,\cdot)\vert^p-\vert \tilde{u}(\tau,\cdot)\vert^p\Vert_{L^2}\leq C\Vert \vert u(\tau,\cdot)-\tilde{u}(\tau,\cdot)\vert\big(\vert u(\tau,\cdot)\vert^{p-1}+\vert\tilde{u}(\tau,\cdot)\vert^{p-1}\big)\Vert_{L^2}\\
&\leq C\Vert u(\tau,\cdot)-\tilde{u}(\tau,\cdot)\Vert_{L^{2p}}\Vert\vert u(\tau,\cdot)\vert^{p-1}+\vert\tilde{u}(\tau,\cdot)\vert^{p-1}\Vert_{L^{\frac{2p}{p-1}}}\\
&\leq C\Vert u(\tau,\cdot)-\tilde{u}(\tau,\cdot)\Vert_{L^{2}}^{\theta(2p)}\Vert \nabla\big( u(\tau,\cdot)-\tilde{u}(\tau,\cdot)\big)\Vert_{L^{2}}^{1-\theta(2p)}\\
&\quad\quad\quad\quad\quad\quad\quad\quad\quad\quad\quad\quad\quad\quad\times\Bigl(\Vert u(\tau,\cdot)\Vert_{L^{2p}}^{p-1}+\Vert \tilde{u}(\tau,\cdot)\Vert_{L^{2p}}^{p-1}\Bigr)\\
&\leq CM[u-\tilde{u}](\tau)\times\Bigl(\Vert u(\tau,\cdot)\Vert_{L^{2}}^{\theta(2p)(p-1)}\Vert\nabla u(\tau,\cdot)\Vert^{(1-\theta(2p))(p-1)}\\
&\quad\quad\quad\quad\quad\quad\quad\quad\quad\quad\quad+\Vert \tilde{u}(\tau,\cdot)\Vert_{L^{2}}^{\theta(2p)(p-1)}\Vert\nabla \tilde{u}(\tau,\cdot)\Vert^{(1-\theta(2p))(p-1)}\Bigr)\\
&\leq CM[u-\tilde{u}](\tau)\bigr(M[u]^{p-1}(\tau)+M[\tilde{u}]^{p-1}(\tau)\bigr)
\end{aligned}
\end{equation}
and
\begin{equation}\label{uuuu2}
\begin{aligned}
\Vert \vert &u(\tau,\cdot)\vert^p-\vert \tilde{u}(\tau,\cdot)\vert^p\Vert_{L^1}\leq C\Vert \vert u(\tau,\cdot)\vert^p-\vert \tilde{u}(\tau,\cdot)\vert^p\Vert_{L^2}\big(\phi(t)-\phi(1)+M\big)^{\frac{n}{2}}\\
&\leq C\tau^{\frac{(m+1)n}{2}}M[u-\tilde{u}](\tau)\bigr(M[u]^{p-1}(\tau)+M[\tilde{u}]^{p-1}(\tau)\bigr),
\end{aligned}
\end{equation}
where $\theta(2p)$ is defined by (\ref{theta2p}).
Taking $k=0$ in (\ref{3.6c}), using Duhamel's principle and (\ref{uuuu1})-(\ref{uuuu2}), we have
\begin{equation}\label{unnl2}
\begin{aligned}
&\Vert \mathcal{N}u(t,\cdot)-\mathcal{N}\tilde{u}(t,\cdot)\Vert_{L^2}\leq C_mt^{-(m+1)\frac{n}{2}+\frac{\sqrt\delta-\mu+1}{2}}\int_1^t\tau^{-\frac{\sqrt\delta-\mu-1}{2}}\\
&\quad\quad\quad\quad\times\big(\Vert \vert u(\tau,\cdot)\vert^p-\vert \tilde{u}(\tau,\cdot)\vert^p\Vert_{L^1}+\tau^{(m+1)\frac{n}{2}}\Vert \vert u(\tau,\cdot)\vert^p-\vert \tilde{u}(\tau,\cdot)\vert^p\Vert_{L^2}\big)d\tau\\
&\leq C_mt^{-(m+1)\frac{n}{2}+\frac{\sqrt\delta-\mu+1}{2}}\\
&\quad\quad\quad\times\int_1^t\tau^{\frac{(m+1)n}{2}-\frac{\sqrt\delta-\mu-1}{2}}
M[u-\tilde{u}](\tau)\bigr(M[u]^{p-1}(\tau)+M[\tilde{u}]^{p-1}(\tau)\bigr)d\tau\\
&\leq C_mt^{\frac{\sqrt\delta}{2}+1}\int_1^t
M[u-\tilde{u}](\tau)\bigr(M[u]^{p-1}(\tau)+M[\tilde{u}]^{p-1}(\tau)\bigr)d\tau.
\end{aligned}
\end{equation}
By taking $k=1$ in (\ref{3.6c}) and (\ref{3.7c}), we obtain the estimates of $\Vert \nabla\big(\mathcal{N}u(t,\cdot)-\mathcal{N}\tilde{u}(t,\cdot)\big)\Vert_{L^2}$ and $\Vert \partial_t\big(\mathcal{N}u(t,\cdot)-\mathcal{N}\tilde{u}(t,\cdot)\big)\Vert_{L^2}$  as
\begin{equation}\label{unnnn2}
\begin{aligned}
&\Vert \nabla\big(\mathcal{N}u(t,\cdot)-\mathcal{N}\tilde{u}(t,\cdot)\big)\Vert_{L^2}\\
&\quad\quad\leq C_mt^{\frac{\sqrt\delta}{2}+1}\int_1^t
M[u-\tilde{u}](\tau)\bigr(M[u]^{p-1}(\tau)+M[\tilde{u}]^{p-1}(\tau)\bigr)d\tau,\\
&\Vert \partial_t\big(\mathcal{N}u(t,\cdot)-\mathcal{N}\tilde{u}(t,\cdot)\big)\Vert_{L^2}\\
&\quad\quad\leq C_mt^{\frac{\sqrt\delta}{2}}\int_1^t
M[u-\tilde{u}](\tau)\bigr(M[u]^{p-1}(\tau)+M[\tilde{u}]^{p-1}(\tau)\bigr)d\tau.
\end{aligned}
\end{equation}
We infer by  (\ref{unnl2}) and (\ref{unnnn2}) that
\begin{equation*}
M[\mathcal{N}u-\mathcal{N}\tilde{u}](t)\leq C_mt^{\frac{\sqrt\delta}{2}+1}\int_1^tM[u-\tilde{u}](\tau)\bigr(M[u]^{p-1}(\tau)+M[\tilde{u}]^{p-1}(\tau)\bigr)d\tau,
\end{equation*}
which indicates that (\ref{kkey2}) holds, where $C_{3T}=C_mT^{\frac{\sqrt\delta}{2}+1}(T-1)$.

Applying the same argument, we can deal with the remaining cases, so no further details will be provided.\\
Case 2: $\frac{\sqrt\delta}{2(m+1)}+\frac{1}{2}-\frac{n}{2}=1$, i.e., $\delta=(m+1)^2(n+1)^2$.\\
Case 3: $0<\frac{\sqrt\delta}{2(m+1)}+\frac{1}{2}-\frac{n}{2}<1$, i.e., $(m+1)^2(n-1)^2<\delta<(m+1)^2(n+1)^2$.\\
Case 4: $\frac{\sqrt\delta}{2(m+1)}+\frac{1}{2}-\frac{n}{2}=0$, $n>1$,  i.e., $\delta=(m+1)^2(n-1)^2$, $n>1$.\\
Case 5: $\frac{\sqrt\delta}{2(m+1)}+\frac{1}{2}-\frac{n}{2}<0$, $n>1$,  i.e., $0<\delta<(m+1)^2(n-1)^2$, $n>1$.

\section{The Proof of Theorem \ref{theorem2}}
\subsection{Test functions and some asymptotic properties}

In this section, we will provide solutions to the conjugate equation of  (\ref{eq})
\begin{equation}\label{conju}
\partial_t^2\varPhi(t,x)-t^{2m}\Delta\varPhi(t,x)-\partial_t\bigl(\frac{\mu}{t}\varPhi(t,x)\bigr)+\frac{\nu^2}{t^2}\varPhi(t,x)=0.
\end{equation}
Inspired by \cite{Tu2019}, we want to express the solution to (\ref{conju}) in terms of Gaussian hypergeometric functions. For $\beta\in \mathbb{R}$, define
\begin{equation}
\Phi_\beta(t,x)=t^{-\beta+1}\psi_\beta\bigl(\frac{(m+1)^2\vert x\vert^2}{t^{2(m+1)}}\bigr),\ t\geq1,
\end{equation} where $\psi_\beta\in C^2\bigl([0,\frac{1}{m+1})\bigr)$. For brevity, let $z=\frac{(m+1)^2\vert x\vert^2}{t^{2(m+1)}}$.

\begin{lemma}\label{le1}
 $\Phi_\beta(t,x)$ is a solution to (\ref{conju}) if and only if $\psi_\beta$ is a solution to
\begin{equation}\label{psi}
\begin{aligned}
z(1-z)\psi_\beta^{\prime\prime}(z)&+\Bigl[\frac{n}{2}-\bigl(\frac{\beta-1}{m+1}+\frac{1}{2(m+1)}+\frac{\mu}{2(m+1)}+1\bigr)z\Bigr]\psi_\beta^{\prime}(z)\\
                                  &-\frac{\beta(\beta+\mu-1)+\nu^2}{4(m+1)^2}\psi_\beta(z)=0.
\end{aligned}
\end{equation}
\end{lemma}
\begin{proof}
Substituting
\begin{equation*}
\begin{aligned}
& \partial_t^2\Phi_\beta(t,x)=\beta(\beta-1)t^{-\beta-1}\psi_\beta(z)+4(m+1)^2t^{-\beta-1}z^2\psi_\beta^{\prime\prime}(z)\notag\\
                          &\quad\quad\quad\quad\quad\quad\quad+\bigl[2(m+1)(\beta-1)+2(m+1)\beta+4(m+1)^2\bigr]t^{-\beta-1}z\psi_\beta^{\prime}(z),\notag\\
&\Delta\Phi_\beta(t,x)=4(m+1)^2zt^{-2m}t^{-\beta-1}\psi_\beta^{\prime\prime}(z)+2n(m+1)^2t^{-2m}t^{-\beta-1}\psi_\beta^{\prime}(z),\\
&\frac{\partial}{\partial t}\bigl(\frac{\mu}{t}\Phi_\beta(t,x)\bigr)=\bigl(-\mu+\mu(-\beta+1)\bigr) t^{-\beta-1}\psi_\beta(z)-2(m+1)\mu t^{-\beta-1}z\psi_\beta^{\prime}(z)
\end{aligned}
\end{equation*}
into (\ref{conju}), we get
\begin{equation*}
\begin{aligned}
t^{-\beta-1}&\Bigl\{\bigl[4(m+1)^2z^2-4(m+1)^2z\bigr]\psi_\beta^{\prime\prime}(z)\\
            &+\bigl[\bigl(2(m+1)(2\beta-1)+4(m+1)^2\bigr)z-2(m+1)^2n+2(m+1)\mu z\bigr]\psi_\beta^{\prime}(z)\\
            &+\bigl[\beta(\beta-1)+\mu+\mu(\beta-1)+\nu^2\bigr] \psi_\beta(z)   \Bigr\}=0,
\end{aligned}
\end{equation*}
which implies (\ref{psi}).

\end{proof}

As is well known, when $\vert z\vert<1$,
\begin{equation}\label{Gauss}
F(a,b,c;z)=\displaystyle\sum_{k=0}^{\infty}\frac{(a)_k(b)_k}{k!(c)_k}z^{k}
\end{equation}
is a solution to the Gauss hypergeometric equation
\begin{equation*}
z(1-z)h^{\prime\prime}(z)+[\gamma-(\alpha+\beta+1)z]h^{\prime}(z)-\alpha\beta h(z)=0,
\end{equation*}
where
$(d)_k=
  \left\{
\begin{aligned}
&1,&k=0,\\
&\displaystyle\prod_{j=1}^{k}(d+j-1),&k\geq0
\end{aligned}
 \right.$
and  $F(a,b,c;z)$ is the Gauss hypergeometric function with parameters $(a, b, c).$

\begin{lemma}\label{le2}
Assume that $\vert z\vert<1$ and $\psi_\beta(z)$ is a solution to (\ref{psi}), then
\begin{equation}\label{psi1}
 \psi_\beta(z) =F(a,b,c;z)=\displaystyle\sum_{k=0}^{\infty}\frac{(a)_k(b)_k}{k!(c)_k}z^{k},
\end{equation}
where $a=\frac{2\beta+\mu-1+\sqrt\delta}{4(m+1)}$, $b=\frac{2\beta+\mu-1-\sqrt\delta}{4(m+1)}$ and $c=\frac{n}{2}$.
\end{lemma}

\begin{proof}
Consider the quadratic equation
\begin{equation*}
s^2-\frac{2\beta+\mu-1}{2(m+1)}s+\frac{\beta(\beta+\mu-1)+\nu^2}{4(m+1)^2}=0,
\end{equation*}
since $\frac{(2\beta+\mu-1)^2}{4(m+1)^2}-\frac{\beta(\beta+\mu-1)+\nu^2}{(m+1)^2}=\frac{\delta}{4(m+1)^2}>0$,  then $a=\frac{2\beta+\mu-1+\sqrt\delta}{4(m+1)}$ and $b=\frac{2\beta+\mu-1-\sqrt\delta}{4(m+1)}$ are the two real roots. Moreover,
\begin{equation*}
  \left\{
\begin{aligned}
&a+b+1=\frac{\beta-1}{m+1}+\frac{1}{2(m+1)}+\frac{\mu}{2(m+1)}+1,\\
&ab=\frac{\beta(\beta+\mu-1)+\nu^2}{4(m+1)^2},
\end{aligned}
 \right.
 \end{equation*}
 then (\ref{psi}) is transformed into
 \begin{equation*}
 z(1-z)\psi_\beta^{\prime\prime}(z)+\bigl[\frac{n}{2}-(a+b+1)z\bigr]\psi_\beta^{\prime}(z)-ab\psi_\beta(z)=0,
 \end{equation*}
 which shows (\ref{psi1}) by (\ref{Gauss}).

\end{proof}

From Lemmas \ref{le1} and \ref{le2}, we know that
\begin{equation}\label{conju2}
\Phi_\beta(t,x)=t^{-\beta+1}\psi_\beta(z)=t^{-\beta+1}F(a,b,c;z)
\end{equation}
is a solution to (\ref{conju}), where $a, b, c$ are given by Lemma \ref{le2}.

The following lemma  gives the asymptotic behavior of $\psi_\beta(z)$ and $\psi_\beta^{\prime}(z)$.
\begin{lemma}\label{le}
(i) If $\frac{1+\sqrt\delta-\mu}{2}<\beta<\frac{(m+1)n-\mu+1}{2}$,  there exists a $C=C(m,n,\beta,\mu,\nu)>1$ such that for any $z\in[0,1)$,
\begin{equation}\label{asy1}
1\leq\psi_\beta(z)\leq C.
\end{equation}
(ii) If $\beta>\frac{(m+1)(n-2)-\mu+1}{2}$, there exists a $C=C(m,n,\beta,\mu,\nu)>1$ such that for any $z\in[0,1)$,
\begin{equation}\label{asy2}
C^{-1}(1-\sqrt z)^{\frac{(m+1)(n-2)-\mu+1-2\beta}{2(m+1)}}\leq\vert \psi_\beta^{\prime}(z)\vert\leq C(1-\sqrt z)^{\frac{(m+1)(n-2)-\mu+1-2\beta}{2(m+1)}}.
\end{equation}
\end{lemma}

\begin{proof}
(i) $\frac{1+\sqrt\delta-\mu}{2}<\beta<\frac{(m+1)n-\mu+1}{2}$ guarantees  that $a>0,\ b>0$ and $c-a-b>0$. By the formula (15.4.20) in \cite{OlLo2010}:
\begin{equation*}
F(a,b,c;1)=\frac{\Gamma(c)\Gamma(c-a-b)}{\Gamma(c-a)\Gamma(c-b)}\ \text{for} \ Re(c-a-b)>0,
\end{equation*}
we deduce that (\ref{asy1}) holds. 

(ii) Note the formula (15.5.1) in \cite{OlLo2010}:
\begin{equation*}
\frac{d}{dz}F(a,b,c;z)=\frac{ab}{c}F(a+1,b+1,c+1;z),
\end{equation*}
so
\begin{equation*}
\psi_\beta^{\prime}(z)=\frac{ab}{c}F(a+1,b+1,\frac{n}{2}+1;z).
\end{equation*}
Then the formula (15.4.23)  in \cite{OlLo2010}:
\begin{equation*}
\displaystyle \lim_{z \to  1^-}\frac{F(a,b,c;z)}{(1-z)^{c-a-b}}=\frac{\Gamma(c)\Gamma(a+b-c)}{\Gamma(a)\Gamma(b)}\ \text{for} \ Re(c-a-b)<0
\end{equation*}
shows that
\begin{align*}
&\displaystyle \lim_{z \to  1^-}\frac{\psi_\beta^{\prime}(z)}{(1-z)^{\frac{n}{2}+1-(a+1)-(b+1)}}=\displaystyle \lim_{z \to  1^-}\frac{2ab}{n}\frac{F(a+1,b+1,\frac{n}{2}+1;z)}{(1-z)^{\frac{n}{2}-a-b-1}}\\
&\quad=\displaystyle \lim_{z \to  1^-}\frac{2ab}{n}\frac{F(a+1,b+1,\frac{n}{2}+1;z)}{(1-z)^{\frac{(m+1)(n-2)-\mu+1-2\beta}{2(m+1)}}}=\frac{\Gamma(\frac{n}{2}+1)\Gamma(a+b+1-\frac{n}{2})}{\Gamma(a+1)\Gamma(b+1)}>0,
\end{align*}
where we used that $\frac{n}{2}+1-(a+1)-(b+1)=\frac{(m+1)(n-2)-2\beta-\mu+1}{2(m+1)}<0 $ because $ \beta>\frac{(m+1)(n-2)-\mu+1}{2}$.  Then
\begin{equation*}
C^{-1}(1- z)^{\frac{(m+1)(n-2)-\mu+1-2\beta}{2(m+1)}}\leq\vert \psi_\beta^{\prime}(z)\vert\leq C(1- z)^{\frac{(m+1)(n-2)-\mu+1-2\beta}{2(m+1)}}
\end{equation*}
holds for any $z\in[0,1)$, therefore (\ref{asy2}) holds by noting  $1+\sqrt z$ is bounded from above and below for $z\in[0,1)$.

\end{proof}

\subsection{Some preliminary estimates}

Before starting to prove  Theorem \ref{theorem2}, we will provide some preliminary results. Let
\begin{align}
H_\beta(t)&=\int_{\mathbb{R}^n}\vert u(t,x)\vert^p\Phi_\beta(t,x)dx,\label{G}\\
I_\beta(t)&=\int_1^t(t-s)sH_\beta(s)ds,\label{H}\\
J_\beta(t)&=\int_1^t(1+s)^{-3}I_\beta(s)ds\label{U},
\end{align}
where $\beta\in(\frac{1+\sqrt\delta-\mu}{2},\frac{(m+1)n-\mu+1}{2})$, $\Phi_\beta(t,x)$ is a solution of (\ref{conju}) as construsted in Section 4.1 and $t\geq1$. We will show that $J_\beta(t)$ blows up before $e^{c\varepsilon^{-p(p-1)}}$, hence, $I_\beta(t)$ and $H_\beta(t)$ as well.
\begin{remark}
The condition $\delta<(m+1)^2n^2$ in Theorem \ref{theorem2} guarantees the existence of  $\beta\in(\frac{1+\sqrt\delta-\mu}{2}, \frac{(m+1)n-\mu+1}{2})$, and as mentioned in Section 1, the range of $\delta$ reflects the competitive relationship between the Fujita indice and the Strauss one.
\end{remark}

 Now we give some preliminary estimates.
\begin{lemma}\label{le4.1}
For any $\beta\in(\frac{1+\sqrt\delta-\mu}{2},\frac{(m+1)n-\mu+1}{2})$ and $t\geq1$,
\begin{equation*}
t^2J_\beta(t)\leq\frac{1}{2}\int_1^t(t-s)^2H_\beta(s)ds.
\end{equation*}
\end{lemma}
\begin{proof}
Differentiating (\ref{H}), we have
\begin{equation*}
I_\beta^{\prime}(t)=\int_1^tsH_\beta(s)ds \quad \text{and} \quad I_\beta^{\prime\prime}(t)=tH_\beta(t).
\end{equation*}
Since $I_\beta(1)=I_\beta^{\prime}(1)=0$, it follows from integration by parts that
\begin{equation*}
\begin{aligned}
\int_1^t(t-s)^2H_\beta(s)ds&=\int_1^t(t-s)^2s^{-1}I_\beta^{\prime\prime}(s)ds=-\int_1^t\frac{\partial}{\partial s}\bigl((t-s)^2s^{-1}\bigr)I_\beta^{\prime}(s)ds\\
                                     &=\int_1^t\frac{\partial^2}{\partial s^2}\bigl((t-s)^2s^{-1}\bigr)I_\beta(s)ds=\int_1^t\frac{2t^2}{s^3}I_\beta(s)ds\geq 2t^2J_\beta(t).
\end{aligned}
\end{equation*}

\end{proof}

\begin{lemma}\label{le4.4}
Suppose that $(u_0,u_1)\in D$  be nonnegative, not identically zero  and $supp(u_0,u_1)\subset\{x:\vert x\vert\leq M \}$ for some $M<\frac{1}{m+1}$, then for any $\beta\in(\frac{1+\sqrt\delta-\mu}{2},\frac{(m+1)n-\mu+1}{2})\cap [1-\mu, +\infty)$,  $\delta<(m+1)^2n^2$ and $t\geq1$, we have
\begin{equation}\label{E1}
\begin{aligned}
&\varepsilon\mathcal{E}_{0,\beta,m}(u_0)+\varepsilon\mathcal{E}_{1,\beta,m}(u_0,u_1)(t-1)+\int_1^t(t-s)H_\beta(s)ds\\
&\quad=\int_{\mathbb{R}^n}u(t,x)\Phi_\beta(t,x)dx+\int_1^ts^{-\beta}\int_{\mathbb{R}^n}u(s,x)\overline{\psi_\beta}\bigl(\frac{(m+1)^2\vert x\vert^2}{s^{2(m+1)}}\bigr)dxds,
\end{aligned}
\end{equation}
where
\begin{equation*}
\begin{aligned}
&\overline{\psi}_\beta(z)=(2\beta+\mu-2)\psi_\beta(z)+4(m+1)z\psi_\beta^{\prime}(z),\\
&\mathcal{E}_{0,\beta,m}(u_0)=\int_{\mathbb{R}^n}u_0(x)\psi_\beta\bigl((m+1)^2\vert x\vert^2\bigr)dx\geq 0,\\
&\mathcal{E}_{1,\beta,m}(u_0,u_1)=\int_{\mathbb{R}^n}u_1(x)\psi_\beta\bigl((m+1)^2\vert x\vert^2\bigr)dx\\
                               &\quad+\int_{\mathbb{R}^n}u_0(x)\Bigl[2(m+1)^3\vert x\vert^2\psi_\beta^{\prime}\bigl((m+1)^2\vert x\vert^2\bigr)+(\mu+\beta-1)\psi_\beta\bigl((m+1)^2\vert x\vert^2\bigr) \Bigr]dx\geq 0.\\
\end{aligned}
\end{equation*}

\end{lemma}
\begin{remark}
Noting that $1-\mu\leq \frac{(m+1)n-\mu+1}{2}$, so  $(\frac{1+\sqrt\delta-\mu}{2},\frac{(m+1)n-\mu+1}{2})\cap [1-\mu, \infty)\neq \varnothing$. We claim that $\mathcal{E}_{0,\beta,m}(u_0)$ and $\mathcal{E}_{1,\beta,m}(u_0,u_1)$ are nonnegative. Indeed, since $\beta\in(\frac{1+\sqrt\delta-\mu}{2},\frac{(m+1)n-\mu+1}{2})$ and
$\psi_\beta\bigl((m+1)^2\vert x\vert^2\bigr)=F(a,b;\frac{n}{2};(m+1)^2\vert x\vert^2)$,
(\ref{asy1}) shows that $\mathcal{E}_{0,\beta,m}(u_0)\geq0$; obviously, $\psi_\beta^{\prime}\bigl((m+1)^2\vert x\vert^2\bigr)=\frac{2ab}{n}F(a+1,b+1;\frac{n}{2};(m+1)^2\vert x\vert^2)>0$,  $\beta\geq1-\mu$ and the nonnegativity of  $u_0,  u_1$ imply that  $\mathcal{E}_{1,\beta,m}(u_0,u_1)\geq0$.
\end{remark}

\renewcommand{\proofname}{\it  Proof of Lemma \ref{le4.4}}
\begin{proof}
Note (\ref{conju}), (\ref{G}) and substitute $\Phi_\beta(t,x)$ into (\ref{def3}), we obtain

\begin{align}
&\int_{\mathbb{R}^n}\bigl(\partial_tu(t,x)\varPhi_\beta(t,x)-u(t,x)\partial_t\varPhi_\beta(t,x)+\frac{\mu}{t}u(t,x)\varPhi_\beta(t,x)\bigr)dx\label{4.13}\\
&\quad=\int_1^tH_\beta(s)ds+\varepsilon\int_{\mathbb{R}^n}\Bigl(-u_0(x)\partial_t\varPhi_\beta(1,x)+\bigl(\mu u_0(x)+u_1(x)\bigr)\varPhi_\beta(1,x)\Bigr)dx.\nonumber
\end{align}
Since
$\partial_t\Phi_\beta(1,x)=(-\beta+1)\psi_\beta\bigl((m+1)^2\vert x\vert^2\bigr)-2(m+1)^3\vert x\vert^2\psi_\beta^{\prime}\bigl((m+1)^2\vert x\vert^2\bigr)$
and $\Phi_\beta(1,x)=\psi_\beta\bigl((m+1)^2\vert x\vert^2\bigr)$, then
\begin{align}\label{cau1}
&\int_{\mathbb{R}^n}\Bigl(-u_0(x)\partial_t\varPhi_\beta(1,x)+\bigl(\mu u_0(x)+u_1(x)\bigr)\varPhi_\beta(1,x)\Bigr)dx\nonumber\\
&\quad =\int_{\mathbb{R}^n}u_1(x)\psi_\beta\bigl((m+1)^2\vert x\vert^2\bigr)dx\\
&\quad\quad+\int_{\mathbb{R}^n}u_0(x)\Bigl[2(m+1)^3\vert x\vert^2\psi_\beta^{\prime}\bigl((m+1)^2\vert x\vert^2\bigr)+(\mu+\beta-1)\psi_\beta\bigl((m+1)^2\vert x\vert^2\bigr)\Bigr]\nonumber\\
&\quad =\mathcal{E}_{1,\beta,m}(u_0,u_1)\geq0\nonumber.
\end{align}
Further calculation leads to
\begin{equation}
\begin{aligned}\label{cau2}
&\int_{\mathbb{R}^n}\bigl(\partial_tu(t,x)\varPhi_\beta(t,x)-u(t,x)\partial_t\varPhi_\beta(t,x)+\frac{\mu}{t}u(t,x)\varPhi_\beta(t,x)\bigr)dx\\
&\quad=\frac{d}{dt}\int_{\mathbb{R}^n}u(t,x)\varPhi_\beta(t,x)dx+t^{-\beta}\int_{\mathbb{R}^n}\overline{\psi_\beta}(z)u(t,x)dx.
\end{aligned}
\end{equation}
If follows from (\ref{4.13})-(\ref{cau2}) that
\begin{equation*}
\begin{aligned}
&\frac{d}{dt}\int_{\mathbb{R}^n}u(t,x)\varPhi_\beta(t,x)dx+t^{-\beta}\int_{\mathbb{R}^n}\overline{\psi_\beta}(z)u(t,x)dx=\int_1^tH_\beta(s)ds+\varepsilon\mathcal{E}_{1,\beta,m}(u_0,u_1).
\end{aligned}
\end{equation*}
A further integration by parts over $[1,t]$ provides that
\begin{equation*}
\begin{aligned}
&\int_{\mathbb{R}^n}u(t,x)\varPhi_\beta(t,x)dx-\varepsilon \mathcal{E}_{0,\beta,m}(u_0)+\int_1^ts^{-\beta}\int_{\mathbb{R}^n}u(s,x)\overline{\psi_\beta}\bigl(\frac{(m+1)^2\vert x\vert^2}{s^{2(m+1)}}\bigr)dxds\\
&\quad=\varepsilon\mathcal{E}_{1,\beta,m}(u_0,u_1)(t-1)+\int_1^t(t-s)H_\beta(s)ds,
\end{aligned}
\end{equation*}
which implies (\ref{E1}).

\end{proof}

For any $q>1$, consider
\begin{equation}
\beta_q=\frac{(m+1)n-\mu+1}{2}-\frac{m+1}{q}
\end{equation}
and let $p^{\prime}$  be the conjugate number of $p$, i.e., $\frac{1}{p}+\frac{1}{p^{\prime}}=1$.
\begin{lemma}\label{le4.5}
Under the assumptions of Lemma \ref{le4.4}, let $p>\frac{2(m+1)}{(m+1)n-\sqrt\delta}$ and when $n=1$, we assume $\mu\geq m.$\\
(i) When $q>p$, there exists a $C=C(m,n,\beta,\mu,\nu)>0$ such that
\begin{equation}\label{E21}
\begin{aligned}
&\varepsilon\mathcal{E}_{0,\beta_q,m}(u_0)+\varepsilon\mathcal{E}_{1,\beta_q,m}(u_0,u_1)(t-1)+\int_1^t(t-s)H_{\beta_q}(s)ds\\
&\quad\leq C t^{-\beta_q+1+\frac{(m+1)n}{p^{\prime}}}\Vert u(t,\cdot)\Vert_{L^{p}(\mathbb{R}^n)}+C\int_1^ts^{-\beta_p+\frac{(m+1)n}{p^{\prime}}}\Vert u(s,\cdot)\Vert_{L^{p}(\mathbb{R}^n)}ds.
\end{aligned}
\end{equation}
(ii) When $q=p$, there exists a $C=C(m,n,\beta,\mu,\nu)>0$ such that
\begin{equation}\label{E22}
\begin{aligned}
&\int_1^t(t-s)H_{\beta_p}(s)ds\leq Ct^{-\beta_p+1+\frac{(m+1)n}{p^{\prime}}}\Vert u(t,\cdot)\Vert_{L^{p}(\mathbb{R}^n)}\\
                    &\quad+C\int_1^ts^{\frac{(m+1)n}{p^\prime}-\beta_p}(\log s)^{\frac{1}{p^\prime}}\Vert u(s,\cdot)\Vert_{L^{p}(\mathbb{R}^n)}ds.
\end{aligned}
\end{equation}
\end{lemma}

\begin{remark} The proof of Lemma \ref{le4.5}  is based on  Lemmas \ref{le} and \ref{le4.4}, so we need to guarantee that
\begin{equation*}\label{check}
\vert z\vert=\frac{(m+1)^2\vert x\vert^2}{t^{2(m+1)}}<1,
\end{equation*}
\begin{equation}\label{check1}
\beta_p=\frac{(m+1)n-\mu+1}{2}-\frac{m+1}{p}\in(\frac{1+\sqrt\delta-\mu}{2},\frac{(m+1)n-\mu+1}{2})
\end{equation}
and
\begin{equation}\label{check2}
\beta_p\geq1-\mu.
\end{equation}
Indeed, the finite propagation speed of waves implies
\begin{equation}\label{finitespeed}
supp\ \text{u}(t,\cdot)\subset\{x: \vert x\vert\leq\phi(t)-\phi(1)+M\},
\end{equation}
note that  $M<\frac{1}{m+1}=\phi(1)$, so $\vert x\vert\leq\frac{t^{m+1}}{m+1}-\frac{1}{m+1}+M<\frac{t^{m+1}}{m+1}$, then $\vert z\vert=\frac{(m+1)^2\vert x\vert^2}{t^{2(m+1)}}<1$. In addition, (\ref{check1})  is equal to $p>\frac{2(m+1)}{(m+1)n-\sqrt\delta}$.
Hence we  need to check (\ref{check2}). In fact,
\begin{equation*}
\begin{aligned}
\beta_p\geq 1-\mu\iff\frac{2(m+1)}{p}\leq(m+1)n+\mu-1.
\end{aligned}
\end{equation*}
Since $p=p_S(n+\frac{\mu}{m+1},m)$, by (\ref{Lambdaa})  we have
\begin{equation*}
\frac{2(m+1)}{p}=p[(m+1)n+\mu-1]-[(m+1)(n-2)+\mu+3]
\end{equation*}
and
\begin{equation*}
\resizebox{0.88\hsize}{!}{$p=\frac{(m+1)(n-2)+\mu+3+\sqrt{[(m+1)(n-2)+\mu+3]^2+8(m+1)[(m+1)n+\mu-1]}}{2[(m+1)n+\mu-1]}, $}
\end{equation*}
then
\begin{align}
&\frac{2(m+1)}{p}\leq(m+1)n+\mu-1\notag\\
&\iff p[(m+1)n+\mu-1]-[(m+1)(n-2)+\mu+3]\leq(m+1)n+\mu-1\notag\\
&\iff(m+1)^2(n^2-2n)+2(m+1)(n\mu+1-\mu)+\mu^2-1\geq0,\label{equal}
\end{align}
this holds obviously for any $n\geq 2$. And when $n=1$,
\begin{equation*}
(\ref{equal}) \iff-(m+1)^2+2(m+1)+\mu^2-1\geq0\iff\mu^2\geq m^2\iff \mu\geq m.
\end{equation*}
\end{remark}

Now we give the proof of Lemma \ref{le4.5}.
\renewcommand{\proofname}{\it  Proof }
\begin{proof}
According to the assumptions on $q$ in (i) and (ii), we know that $\beta_q\geq 1-\mu$ and $\beta_q\in(\frac{1+\sqrt\delta-\mu}{2},\frac{(m+1)n-\mu+1}{2})$.  By Lemma \ref{le4.4}, we have
\begin{equation*}
\begin{aligned}
\varepsilon &\mathcal{E}_{0,\beta_q,m}(u_0)+\varepsilon \mathcal{E}_{1,\beta_q,m}(u_0,u_1)(t-1)+\int_1^t(t-s)H_{\beta_q}(s)ds\\
&=\int_{\mathbb{R}^n}u(t,x)\Phi_{\beta_q}(t,x)dx+\int_1^ts^{-\beta_q}\int_{\mathbb{R}}u(s,x)\overline{\psi_{\beta_q}}\bigl(\frac{(m+1)^2\vert x\vert^2}{s^{2(m+1)}}\bigr)dxds\\
&:=\Omega_{1}(t)+\int_1^t\Omega_{2}(s)ds,
\end{aligned}
\end{equation*}
where
\begin{equation*}
\begin{aligned}
&\Omega_{1}(t)=\int_{\mathbb{R}^n}u(t,x)\Phi_{\beta_q}(t,x)dx,\\
&\Omega_{2}(t)=t^{-\beta_q}\int_{\mathbb{R}}u(t,x)\overline{\psi}_{\beta_q}\bigl(\frac{(m+1)^2\vert x\vert^2}{t^{2(m+1)}}\bigr)dx
\end{aligned}
\end{equation*}
with $z=\frac{(m+1)^2\vert x\vert^2}{t^{2(m+1)}}$.

H$\ddot{\text{o}}$lder$^{\prime}$s inequality implies that
\begin{equation}\label{I1}
\begin{aligned}
\Omega_{1}(t)\leq\bigl(\int_{\mathbb{R}^n}\vert u(t,x)\vert^pdx\bigr)^{\frac{1}{p}}\bigl(\int_{\mathbb{R}^n}\vert \Phi_{\beta_q}(t,x)\vert^{p^{\prime}}dx\bigr)^{\frac{1}{p^{\prime}}}.
\end{aligned}
\end{equation}
By (i) of Lemma \ref{le},
\begin{equation*}
\begin{aligned}
\Phi_{\beta_q}(t,x)=t^{-\beta_q+1}\psi_{\beta_q}(z)\leq Ct^{-\beta_q+1},
\end{aligned}
\end{equation*}
so
\begin{equation*}
\begin{aligned}
\Omega_{1}(t)&\leq C\Vert u(t,\cdot)\Vert_{L^p(\mathbb{R}^n)}t^{-\beta_q+1}\bigl(\int_{\vert x\vert\leq\phi(t)-\phi(1)+M}dx\bigr)^{\frac{1}{p^{\prime}}}\\
                          &\leq C t^{-\beta_q+1+\frac{(m+1)n}{p^{\prime}}}\Vert u(t,\cdot)\Vert_{L^p(\mathbb{R}^n)}.
\end{aligned}
\end{equation*}

By H$\ddot{\text{o}}$lder$^{\prime}$s inequality,
\begin{equation*}
\begin{aligned}
\Omega_{2}(t)&\leq t^{-\beta_q}\Vert u(t,\cdot)\Vert_{L^p(\mathbb{R}^n)}\Bigl[\int_{\mathbb{R}^n}\bigl(\overline{\psi}_{\beta_q}(z)\bigr)^{p^\prime}dx\Bigr]^{\frac{1}{p^{\prime}}},
\end{aligned}
\end{equation*}
by  Lemma \ref{le}, $\vert (2\beta+\mu-2)\psi_{\beta_q}(z)\vert\leq C$ and
\begin{equation*}
\begin{aligned}
\vert\psi_{\beta_q}^{\prime}(z)\vert\leq C(1-\sqrt z)^{\frac{(m+1)(n-2)-\mu+1-2\beta_q}{2(m+1)}},
\end{aligned}
\end{equation*}
then
\begin{equation*}
\begin{aligned}
&\vert\overline{\psi}_{\beta_q}(z\bigr)\vert=\vert (2\beta+\mu-2)\psi_{\beta_q}(z)+4(m+1)z\psi_{\beta_q}^{\prime}(z)\vert\\
&\leq C \Bigl(1-\frac{(m+1)\vert x\vert}{t^{m+1}}\Bigr)^{\frac{(m+1)(n-2)-\mu+1-2\beta_q}{2(m+1)}},
\end{aligned}
\end{equation*}
so
\begin{equation}\label{I2}
\resizebox{   0.9\hsize}{!}{$  \begin{aligned}
&\Omega_{2}(t)\leq C t^{-\beta_q}\Vert u(t,\cdot)\Vert_{L^p(\mathbb{R}^n)}\Bigl[ \int_{\vert x\vert\leq \phi(t)-\phi(1)+M}\Bigl(1-\frac{(m+1)\vert x\vert}{t^{m+1}}\Bigr)^{\frac{(m+1)(n-2)-\mu+1-2\beta_q}{2(m+1)}p^\prime}dx\Bigr]^{\frac{1}{p^{\prime}}}\\
                           &\quad=C t^{-\beta_q}\Vert u(t,\cdot)\Vert_{L^p(\mathbb{R}^n)}\Bigl( \int_0^{\frac{(m+1)(\phi(t)-\phi(1)+M)}{t^{m+1}}}\bigl(1-\rho\bigr)^{-\frac{p^\prime}{q^\prime}}\bigl( \frac{t^{m+1}}{m+1}\rho\bigr)^{n-1}\frac{t^{m+1}}{m+1}d\rho\Bigl)^{\frac{1}{p^\prime}}\\
                           &\quad\leq Ct^{-\beta_q+\frac{(m+1)n}{p^\prime}}\Vert u(t,\cdot)\Vert_{L^p(\mathbb{R}^n)}\Bigl( \int_0^{\frac{(m+1)(\phi(t)-\phi(1)+M)}{t^{m+1}}}\bigl(1-\rho\bigr)^{-\frac{p^\prime}{q^\prime}}d\rho\Bigl)^{\frac{1}{p^\prime}},
\end{aligned}      $}
\end{equation}
here we calculated that $\frac{(m+1)(n-2)-\mu+1-2\beta_q}{2(m+1)}p^{\prime}=-\frac{p^{\prime}}{q^{\prime}}$ and $\rho\leq 1$  because of $M\leq\frac{1}{m+1}=\phi(1),$
\begin{equation*}
\begin{aligned}
\rho=\frac{(m+1)r}{t^{m+1}}=\frac{(m+1)\vert x\vert}{t^{m+1}}\leq(m+1)\frac{\phi(t)-\phi(1)+M}{t^{m+1}}\leq (m+1)\frac{\phi(t)}{t^{m+1}}=1.
\end{aligned}
\end{equation*}
When $q>p$, we have $-\frac{p^{\prime}}{q^{\prime}}<-1$, so
\begin{equation*}
\begin{aligned}
\int_0^{\frac{(m+1)(\phi(t)-\phi(1)+M)}{t^{m+1}}}&\bigl(1-\rho\bigr)^{-\frac{p^\prime}{q^\prime}}d\rho\leq C\Bigl(1-\frac{(m+1)(\phi(t)-\phi(1)+M)}{t^{m+1}}\Bigr)^{1-\frac{p^{\prime}}{q^{\prime}}}\\
                                                                                                   &=C \Bigl( \frac{1-M(m+1)}{t^{m+1}}\Bigl)^{1-\frac{p^{\prime}}{q^{\prime}}}\leq C \Bigl( \frac{1}{t^{m+1}}\Bigl)^{1-\frac{p^{\prime}}{q^{\prime}}};
\end{aligned}
\end{equation*}
while $q=p$,
\begin{equation*}
\begin{aligned}
\int_0^{\frac{(m+1)(\phi(t)-\phi(1)+M)}{t^{m+1}}}&\bigl(1-\rho\bigr)^{-\frac{p^\prime}{q^\prime}}d\rho=\int_0^{\frac{(m+1)(\phi(t)-\phi(1)+M)}{t^{m+1}}}\frac{1}{1-\rho}d\rho\\
&=\log\frac{t^{m+1}}{1-(m+1)M}\leq C\log t.
\end{aligned}
\end{equation*}
Then,
\begin{equation}\label{I22}
\Omega_{2}(t) \leq
  \left\{
\begin{aligned}
&C t^{-\beta_p+\frac{(m+1)n}{p^{\prime}}}\Vert u(t,\cdot)\Vert_{L^p(\mathbb{R}^n)},\quad &q>p,\\
&C t^{-\beta_q+\frac{(m+1)n}{p^{\prime}}}\bigl(\log t\bigr)^{\frac{1}{p^\prime}}\Vert u(t,\cdot)\Vert_{L^p(\mathbb{R}^n)},\quad &q=p.
\end{aligned}
 \right.
  \end{equation}
Therefore,
\begin{equation}\label{I23}
\int_1^t\Omega_{2}(s)ds \leq
  \left\{
\begin{aligned}
&C \int_1^ts^{-\beta_p+\frac{(m+1)n}{p^{\prime}}}\Vert u(s,\cdot)\Vert_{L^p(\mathbb{R}^n)}ds,\quad &q>p,\\
&C \int_1^ts^{-\beta_p+\frac{(m+1)n}{p^{\prime}}}\bigl(\log s\bigr)^{\frac{1}{p^\prime}}\Vert u(s,\cdot)\Vert_{L^p(\mathbb{R}^n)}ds,\quad &q=p.
\end{aligned}
 \right.
  \end{equation}

Consequently, (\ref{E21}) and (\ref{E22}) follow from  (\ref{I1}), (\ref{I23}) and the nonnegativity of $\mathcal{E}_{0,\beta,m}(u_0)$ and $\mathcal{E}_{1,\beta,m}(u_0,u_1)$.
\end{proof}

\subsection{The proof of Theorem \ref{theorem2}.} Now, we will derive a second-order ordinary differential inequality satisfied by $J_\beta(t)$, and then utilize the following result  proved by Zhou \cite{Zhou2014} to obtain the   lifespan estimate.
\begin{lemma}\label{le4.7} (\cite{Zhou2014})
Let $\sigma_0\geq0$ and $U\in C^2\bigl([\sigma_0, \infty)\bigr)$ be nonnegative function. Assume there exists positive numbers $c$, $C$ and $C^\prime$ such that
\begin{equation}\label{odr2}
\left\{
\begin{aligned}
& U^{\prime\prime}(\sigma)+2U^{\prime}(\sigma)\geq c\sigma^{1-p}\bigl(U(\sigma)\bigr)^p,\\
&U(\sigma)\geq C\varepsilon^p\sigma,\\
&U^{\prime}(\sigma)\geq C^\prime \varepsilon^p.
\end{aligned}
 \right.
 \end{equation}
Then there exists a  $\varepsilon_0>0$ such that for any $\varepsilon\in(0,\varepsilon_0)$, $U$ blows up before $C^{\prime\prime}\varepsilon^{-p(p-1)}$ for some $C^{\prime\prime}>0$.
\end{lemma}

Now we prove Theorem \ref{theorem2}.

\begin{proof}
For $\sigma>0$,  $\beta_{p+\sigma}=\frac{(m+1)n-\mu+1}{2}-\frac{m+1}{p+\sigma}\in(\frac{1+\sqrt\delta-\mu}{2},\frac{(m+1)n-\mu+1}{2})$ and $\beta_{p+\sigma}\geq 1-\mu$.
From (\ref{conju2}) and (\ref{G}), by Lemma \ref{le}, there exists a $C>0$ such that
\begin{equation*}
C^{-1}t^{(\beta_p-1)\frac{1}{p}}\bigl(H_{\beta_p}(t)\bigr)^{\frac{1}{p}}\leq\Vert u(t,\cdot)\Vert_{L^p(\mathbb{R}^n)}\leq Ct^{(\beta_p-1)\frac{1}{p}}\bigl(H_{\beta_p}(t)\bigr)^{\frac{1}{p}}.
\end{equation*}
Note
\begin{equation*}
\begin{aligned}
-\beta_{p+\sigma}+1+\frac{(m+1)n}{p^\prime}+\frac{1}{p}(\beta_p-1)=\frac{(m+1)n-\beta_p+1}{p^\prime}-1
\end{aligned}
\end{equation*}
and
\begin{equation*}
\begin{aligned}
&\frac{(m+1)n-\beta_p+1}{p^\prime}=\frac{1}{p}+1\\
&+\frac{1}{p^2}\frac{\bigl((m+1)n-1+\mu\bigr)p^2-\bigl((m+1)(n-2)+\mu+3 \bigr)p-2(m+1)}{2}=\frac{1}{p}+1,\\
\end{aligned}
\end{equation*}
it follows from  (i) of Lemma \ref{le4.5} that
\begin{equation}\label{E11}
\begin{aligned}
&\varepsilon\mathcal{E}_{0,\beta_{p+\sigma},m}(u_0)+\varepsilon\mathcal{E}_{1,\beta_{p+\sigma},m}(u_0,u_1)(t-1)\\
&\quad\leq C t^{-\beta_{p+\sigma}+1+\frac{(m+1)n}{p^\prime}}\Vert u(t,\cdot)\Vert_{L^p(\mathbb{R}^n)}+C\int_1^ts^{-\beta_p+\frac{(m+1)n}{p^\prime}}\Vert u(s,\cdot)\Vert_{L^p(\mathbb{R}^n)}ds\\
&\quad\leq Ct^{-\beta_{p+\sigma}+1+\frac{(m+1)n}{p^\prime}+(\beta_p-1)\frac{1}{p}}\bigl(H_{\beta_p}(t)\bigr)^{\frac{1}{p}}\\
&\quad\quad\quad+C\int_1^ts^{-\beta_p+\frac{(m+1)n}{p^\prime}+(\beta_p-1)\frac{1}{p}}\bigl(H_{\beta_p}(t)\bigr)^{\frac{1}{p}}ds\\
&\quad=Ct^{1+\frac{1}{p}+\beta_p-\beta_{p+\sigma}}
\bigl(H_{\beta_p}(t)\bigr)^{\frac{1}{p}}+C\int_1^ts^{\frac{1}{p}}
\bigl(H_{\beta_p}(t)\bigr)^{\frac{1}{p}}ds.
\end{aligned}
\end{equation}
Integrating (\ref{E11}) on $[1,t]$ and using  H$\ddot{\text{o}}$lder$^{\prime}s$ inequality, we have
\begin{equation}\label{E12}
\begin{aligned}
&\varepsilon\mathcal{E}_{0,\beta_{p+\sigma},m}(u_0)(t-1)+\varepsilon\mathcal{E}_{1,\beta_{p+\sigma},m}(u_0,u_1)(\frac{t^2}{2}-t+\frac{1}{2})\\
&\qquad\leq C\int_1^ts^{1+\frac{1}{p}+\beta_p-\beta_{p+\sigma}}\bigl(H_{\beta_p}(s)\bigr)^{\frac{1}{p}}ds+C\int_1^t(t-s)s^{\frac{1}{p}}\bigl(H_{\beta_p}(s)\bigr)^{\frac{1}{p}}ds\\
&\qquad\leq C\Bigl(\int_1^tsH_{\beta_p}(s)ds\Bigr)^{\frac{1}{p}}(t^{p^\prime+1})^{\frac{1}{p^\prime}}=Ct^{1+\frac{1}{p^\prime}}\bigl(I_{\beta_p}^{\prime}(t)\bigr)^{\frac{1}{p}},
\end{aligned}
\end{equation}
which implies that
\begin{equation}\label{E13}
\begin{aligned}
I_{\beta_p}^{\prime}(t)&\geq Ct^{-p(2-\frac{1}{p})}\varepsilon^p\Bigl(\mathcal{E}_{0,\beta_{p+\sigma},m}(u_0)(t-1)+\mathcal{E}_{1,\beta_{p+\sigma},m}(u_0,u_1)\frac{t^2-2t+1}{2}\Bigr)^{p}\\
                             &\geq C\varepsilon^pt, \quad \forall t>1,
\end{aligned}
\end{equation}
then
\begin{equation}\label{E13}
\begin{aligned}
I_{\beta_p}(t)\geq C\varepsilon^pt^2, \quad \forall t>1.
\end{aligned}
\end{equation}
By the definition of $J_{\beta_p}(t)$, $(1+t)^3J_{\beta_p}^\prime(t)=I_{\beta_p}(t)$, so
\begin{equation}\label{57}
\begin{aligned}
&J_{\beta_p}^\prime(t)\geq C(1+t)^{-3}\varepsilon^pt^2\geq C(1+t)^{-1}\varepsilon^p,
\end{aligned}
\end{equation}
then
\begin{equation}\label{56}
\begin{aligned}
J_{\beta_p}(t)\geq C\varepsilon^p\log (1+t).
\end{aligned}
\end{equation}

By (ii) of Lemma \ref{le4.5},
\begin{equation}\label{E14}
\begin{aligned}
&\int_1^t(t-s)H_{\beta_p}(s)ds\leq Ct^{-\beta_p+1+\frac{(m+1)n}{p^{\prime}}}\Vert u(t,\cdot)\Vert_{L^{p}(\mathbb{R}^n)}\\
                    &\qquad\quad\quad+C\int_1^ts^{\frac{(m+1)n}{p^\prime}-\beta_p}(\log s)^{\frac{1}{p^\prime}}\Vert u(s,\cdot)\Vert_{L^{p}(\mathbb{R}^n)}ds\\
                    &\qquad\leq Ct^{-\beta_p+1+\frac{(m+1)n}{p^{\prime}}+\frac{\beta_p-1}{p}}\bigl(H_{\beta_p}(t)\bigr)^{\frac{1}{p}}\\
                     &\quad\quad\quad +C\int_1^ts^{\frac{(m+1)n}{p^\prime}-\beta_p}(\log s)^{\frac{1}{p^\prime}}s^{\frac{\beta_p-1}{p}}\bigl(H_{\beta_p}(s)\bigr)^{\frac{1}{p}}ds\\
                     &\qquad=Ct^{1+\frac{1}{p}}\bigl(H_{\beta_p}(t)\bigr)^{\frac{1}{p}}+C\int_1^ts^{\frac{1}{p}}(\log s)^{\frac{1}{p^\prime}}\bigl(H_{\beta_p}(s)\bigr)^{\frac{1}{p}}ds.
\end{aligned}
\end{equation}
Integrating (\ref{E14}) on $[1,t]$ and using H$\ddot{\text{o}}$lder$^{\prime}s$ inequality give that
\begin{equation}\label{E15}
\begin{aligned}
&\frac{1}{2}\int_1^t(t-s)^2H_{\beta_p}(s)ds\\
&\quad\leq C\int_1^ts^{1+\frac{1}{p}}\bigl(H_{\beta_p}(s)ds\bigr)^{\frac{1}{p}}+C\int_1^t(t-s)s^{\frac{1}{p}}(\log s)^{\frac{1}{p^\prime}}\bigl(H_{\beta_p}(s)\bigr)^{\frac{1}{p}}ds\\
&\quad\leq C(I_{\beta_p}^\prime(t))^{\frac{1}{p}}(1+t)^{\frac{2p-1}{p}}\bigl(\log (t+1)\bigr)^{\frac{1}{p^\prime}},
\end{aligned}
\end{equation}
then Lemma \ref{le4.1} shows that
\begin{equation*}
\begin{aligned}
t^2J_{\beta_p}(t)\leq\frac{1}{2}\int_1^t(t-s)^2H_{\beta_p}(s)ds\leq C(I_{\beta_p}^\prime(t))^{\frac{1}{p}}(1+t)^{\frac{2p-1}{p}}\bigl(\log (t+1)\bigr)^{\frac{1}{p^\prime}},
\end{aligned}
\end{equation*}
i.e.,
\begin{equation}\label{54}
\begin{aligned}
\bigl(\log (t+1)\bigr)^{1-p}\bigl(J_{\beta_p}(t)\bigr)^p&\leq CI_{\beta_p}^\prime(t)(1+t)^{2p-1}t^{-2p}\leq CI_{\beta_p}^\prime(t)(1+t)^{-1}.
\end{aligned}
\end{equation}
By the definition of $J_{\beta_p}(t)$,
\begin{equation}\label{55}
3(1+t)^2J_{\beta_p}^\prime(t)+(1+t)^3J_{\beta_p}^{\prime\prime}(t)=I_{\beta_p}^\prime(t),
\end{equation}
combining (\ref{54}), we have
\begin{equation}\label{55}
\begin{aligned}
\bigl(\log (t+1)\bigr)^{1-p}\bigl(J_{\beta_p}(t)\bigr)^p\leq 3C(1+t)J_{\beta_p}^\prime(t)+C(1+t)^2J_{\beta_p}^{\prime\prime}(t).
\end{aligned}
\end{equation}

Let $1+t=e^{\tau}$ and $J_{0}(\tau)=J_{\beta_p}(t)=J_{\beta_p}(e^\tau-1)$, then

\begin{equation}\label{56}
  \left\{
\begin{aligned}
&J_0^{\prime}(\tau)=J_{\beta_p}^\prime(t)(1+t),\\
&J_0^{\prime\prime}(\tau)=J_{\beta_p}^{\prime\prime}(t)(1+t)^2+J_{\beta_p}^\prime(t)(1+t).
\end{aligned}
 \right.
  \end{equation}
It follows from (\ref{57}), (\ref{55}) and  (\ref{56})  that

\begin{equation}\label{odr1}
\left\{
\begin{aligned}
&J_0^{\prime\prime}(\tau)+2J_0^{\prime}(\tau)\geq C\bigl(J_0(\tau)\bigr)^p\tau^{1-p},\\
&J_0(\tau)\geq C\varepsilon^p\tau,\\
&J_0^{\prime}(\tau)\geq C\varepsilon^p,
\end{aligned}
 \right.
 \end{equation}
by (ii) of Lemma \ref{le4.7} , we deduced that $J_0(\tau)$ blows up before $\tau=C\varepsilon^{-p(p-1)}$, so $J_{\beta_p}(t)$ blows up before $t=e^{C\varepsilon^{-p(p-1)}}$.

\end{proof}

\section{The Proof of Theorem \ref{theorem3}}
\subsection{Simplification of the problem}

In the case of $\delta=1$, we claim that when $n\geq 3$, $p_S(n+\frac{\mu}{m+1},m)>p_F((m+1)n+\frac{\mu-1-\sqrt\delta}{2})$ holds for any $m\geq0$; and for $n=2$, $p_S(n+\frac{\mu}{m+1},m)>p_F((m+1)n+\frac{\mu-1-\sqrt\delta}{2})$ holds if and only if $\mu\in [0, a(m))$, where $a(m)$ is the positive root of the quadratic equation
\begin{equation}\label{ineqequal1}
\eta^2-[8(m+1)^3-8(m+1)^2+2(m+1)-1]\eta +8(m+1)-8(m+1)^2-2=0.
\end{equation}
Indeed,
\begin{align}
&p_S(n+\frac{\mu}{m+1},m)>p_F((m+1)n+\frac{\mu-1-\sqrt\delta}{2})\nonumber\\
&\iff\resizebox{0.88\hsize}{!}{$\frac{(m+1)(n-2)+\mu+3+\sqrt{[(m+1)(n-2)+\mu+3]^2+8(m+1)[(m+1)n+\mu-1]}}{2[(m+1)n+\mu-1]} $}\nonumber\\
&\quad\quad>1+\frac{2}{(m+1)n+\frac{\mu}{2}-1}\label{ineqequal2}\\
&\iff[(m+1)n-2(m+1)-1]\mu^2\nonumber\\
&\qquad\quad+[2(m+1)^3n^2+(m+1)^2n^2-6(m+1)^2n-2(m+1)n+6(m+1)-1]\mu\nonumber\\
&\qquad\quad+2(m+1)^3n^3-4(m+1)^3n^2-2(m+1)^2n^2+8(m+1)^2n\nonumber\\
&\qquad\quad\quad-2(m+1)n-4(m+1)+2>0.\nonumber
\end{align}
We find that
\begin{equation*}
\begin{aligned}
&(m+1)n-2(m+1)-1>0,\\
&2(m+1)^3n^2+(m+1)^2n^2-6(m+1)^2n-2(m+1)n+6(m+1)-1>0
\end{aligned}
\end{equation*}
and
\begin{equation*}
\begin{aligned}
&2(m+1)^3n^3-4(m+1)^3n^2-2(m+1)^2n^2+8(m+1)^2n\\
&\quad\quad\quad-2(m+1)n-4(m+1)+2>0
\end{aligned}
\end{equation*}
hold for $n\geq3$ and any $m\geq0$, so (5.2) is obviously valid since $\mu\geq0$. And when $n=2$, it is easy to see that for (\ref{ineqequal2}) to hold, it is sufficient $\mu \in[0,a(m))$, where $a(m)$ is the positive root of (5.1). Therefore,  in view of Theorem  \ref{theorem2}, we only need to consider the cases of $n=2$ with $\mu\geq a(m)$  and $n=1$ with $p=p_F(m+\frac{\mu}{2})$.

We introduce the change of variables
\begin{equation}
u(t,x)=t^{-\frac{\mu}{2}}w(t,x),
\end{equation}
note $\delta=1$, (\ref{eq}) can be transformed into
\begin{equation}\label{changeeq}
  \left\{
\begin{aligned}
&\partial_t^2w-t^{2m}\Delta w=t^{-\frac{\mu}{2}(p-1)}|w|^p,\ &t\geq1, x\in \mathbb{R}^n,\\
&w(1,x)=w_0(x), \ \partial_tw(1,x)=w(1,x),               &x\in\mathbb{R}^n,
\end{aligned}
 \right.
  \end{equation}
where $w_0(x)=\varepsilon u_0(x)$, $w(1,x)=\frac{\mu}{2}\varepsilon u_0(x)+\varepsilon u_1(x)$. Now, we prove that the solution to (\ref{changeeq}) blows up in finite time, and consequently, so does (\ref{eq}).

\subsection{The Proof of Theorem \ref{theorem3}}The proof is based on the following ordinary differential inequalities, which is called Kato$^\prime$s lemma.
\begin{lemma}(\cite{Zh2006})\label{katolemma}
Let $p>1$, $a\geq1$,  and $(p-1)a=q-2$. Supposed $F\in C^2[0,T)$ satisfies, when $t\geq T_0>0$,\\
(a) $F(t)\geq K_0(t+R)^a$,\\
(b) $F^{\prime\prime}(t)\geq K_1(t+R)^{-q}F^p(t)$,\\
with some positive constants $K_0, K_1, T_0$ and $R$. Fixing $K_1$, there exists a positive constant $c_0$, independent of $R$ and $T_0$ such that if $K_0\geq c_0$, then $T<\infty.$
\end{lemma}

Let
\begin{equation}\label{F}
F(t)=\int_{\mathbb{R}^n}w(t,x)dx,
\end{equation}
by H$\ddot{\text{o}}$lder$^{\prime}$s inequality and the finite propagation speed of waves,
\begin{equation}
\begin{aligned}
\vert F(t)\vert^p&\leq \int_{\vert x\vert\leq \phi(t)-\phi(1)+M}\vert w(t,x)\vert^pdx\Bigl(\int_{\vert x\vert\leq \phi(t)-\phi(1)+M}dx\Bigr)^{p-1}\\
                 &\leq C (t+M)^{(m+1)n(p-1)}\int_{\mathbb{R}^n}\vert w(t,x)\vert^pdx,
\end{aligned}
\end{equation}
then
\begin{equation}\label{prime2F}
\begin{aligned}
F^{\prime\prime}(t)&=t^{-\frac{\mu}{2}(p-1)}\int_{\mathbb{R}^n}\vert w(t,x)\vert^pdx\geq C (t+M)^{-\frac{\mu}{2}(p-1)-(m+1)n(p-1)}\vert F(t)\vert^p.
 \end{aligned}
\end{equation}

By \cite{Ben2022}, we know that
\begin{equation*}\label{lambda}
\lambda(t)=\Bigl(\frac{1}{m+1}\Bigr)^{\frac{1}{2(m+1)}}t^{\frac{1}{2}}K_{\frac{1}{2(m+1)}}\bigl(\frac{t^{m+1}}{m+1}\bigr)
\end{equation*}
 is a solution to
\begin{equation*}
\lambda^{\prime\prime}(t)-t^{2m}\lambda(t)=0,\ t\geq1,
\end{equation*}
where
\begin{equation*}\label{defK}
K_l(z)=\int_0^\infty e^{-z\cosh y}\cosh(ly)dy, \ l\in\mathbb{R}
\end{equation*}
is the modified Bessel function of  order $l$, which is a solution to
\begin{equation*}\label{bessel}
z^2\frac{d^2}{dz^2}w(z)+z\frac{d}{dz}w(z)-(z^2+l^2)w(z)=0.
\end{equation*}

Let
\begin{equation}
G(t)=\int_{\mathbb{R}^n}w(t,x)\psi(t,x)dx,
\end{equation}
where $\psi(t,x)=\lambda(t)\varphi(x)$ and
\begin{equation*}
 \varphi(x)=
 \begin{cases}
  \int_{S^{n-1}}e^{x\cdot \omega}d \omega, &n\geqslant 2,\\
  e^{x}+e^{-x},& n=1,
 \end{cases}
\end{equation*}
so, $\partial_t^2\psi-t^{2m}\Delta \psi=0.$ Furthermore, about $\lambda(t)$, $\varphi(x)$ and $G(t)$ we have
\begin{lemma}(\cite{Ben2022})\label{lambdaprop}
There exists $C>0$ such that for any $t\in[1,\infty)$,
\begin{equation}\label{equalprop}
C^{-1}t^{-\frac{m}{2}}e^{-\phi(t)}\leq\lambda(t)\leq Ct^{-\frac{m}{2}}e^{-\phi(t)}.
\end{equation}
\end{lemma}
\begin{lemma}(\cite{Zh2006})\label{lemma3.3}
Let $r>1$, there exists a $C=C(m,n,r,M)>0$ such that
\begin{equation}
\int_{\vert x\vert\leq \phi(t)-\phi(1)+M}\varphi^{r}(x)dx\leq Ce^{r\phi(t)}\bigl(1+\phi(t)\bigr)^{\frac{(2-r)(n-1)}{2}}.
\end{equation}
\end{lemma}
\begin{lemma}\label{G1}
Let $w$ be a solution to (\ref{changeeq}) and the assumptions in Theorem \ref{theorem3} hold, then there exists a $T_0=T_0(m, n)>1$ such that
\begin{equation}
G(t)\geq C t^{-m}, \ \mbox{for all}  \ t\geq T_0.
\end{equation}
\end{lemma}
\begin{remark}
The proof of Lemma \ref{G1} is quite similar to  Lemma 3.3 in \cite{Ben2022} and Lemma 3.2 in \cite{Ha2021}, so it is not repeated here.
\end{remark}
Utilizing H$\ddot{\text{o}}$lder$^{\prime}$s inequality and Lemmas \ref{lambdaprop}-\ref{G1}, for $t\geq T_0$, we have
\begin{equation*}
\begin{aligned}
\vert G(t)\vert^p&\leq\int_{\mathbb{R}^n}\vert w(t,x)\vert^pdx\Bigl(\int_{\mathbb{R}^n}\vert\psi(t,x)\vert^{\frac{p}{p-1}}dx\Bigr)^{p-1}\\
                 &\leq C\lambda^p(t)e^{p\phi(t)}(1+\phi(t))^{\frac{(p-2)(n-1)}{2}}\int_{\mathbb{R}^n}\vert w(t,x)\vert^pdx\\
                 &\leq C(t+M)^{-\frac{mp}{2}}(1+\phi(t))^{\frac{(p-2)(n-1)}{2}}\int_{\mathbb{R}^n}\vert w(t,x)\vert^pdx,
\end{aligned}
\end{equation*}
therefore,
\begin{equation}\label{Fprimeprime}
\begin{aligned}
F^{\prime\prime}(t)&=t^{-\frac{\mu}{2}(p-1)}\int_{\mathbb{R}^n}\vert w(t,x)\vert^pdx\\
                   &\geq t^{-\frac{\mu}{2}(p-1)}(t+M)^{\frac{mp}{2}}(1+\phi(t))^{-\frac{(p-2)(n-1)}{2}}\vert G(t)\vert^p\\
                   &\geq C(t+M)^{-\frac{\mu}{2}(p-1)-\frac{mp}{2}-\frac{(m+1)(p-2)(n-1)}{2}},
\end{aligned}
\end{equation}
integrating twice we arrive at
\begin{equation}\label{primeprime}
F(t)\geq C(t+M)^{\max\{-\frac{\mu}{2}(p-1)-\frac{mp}{2}-\frac{(m+1)(n-1)(p-2)}{2}+2, 1\}}.
\end{equation}

In the case of $n=2$, since $p=p_F(2(m+1)+\frac{\mu}{2}-1)$ and $\mu\geq a(m)$,  we get
\begin{equation*}
\begin{aligned}
-\frac{\mu}{2}(p-1)-\frac{mp}{2}-\frac{(m+1)(n-1)(p-2)}{2}+2=\frac{3}{2}-\frac{\mu}{2[2(m+1)+\frac{\mu}{2}-1]}\leq1,
\end{aligned}
\end{equation*}
so by (\ref{primeprime}), $F(t)\geq C(t+M)$, substituting it into (\ref{prime2F}) gives
\begin{equation*}
\begin{aligned}
F^{\prime\prime}(t)\geq C (t+M)^{-\frac{4(m+1)+\mu}{2(m+1)+\frac{\mu}{2}-1}+1+\frac{2}{2(m+1)+\frac{\mu}{2}-1}}=C(t+M)^{-1},
\end{aligned}
\end{equation*}
integrating twice and combining (\ref{prime2F}), we get
\begin{equation}\label{kato}
  \left\{
\begin{aligned}
&F^{\prime\prime}(t)\geq C (t+M)^{-\frac{4(m+1)+\mu}{2(m+1)+\frac{\mu}{2}-1}}\vert F(t)\vert^p, \ t\geq T_0,\\
&F(t)\geq C (t+M)\log(t+M).
\end{aligned}
 \right.
  \end{equation}
Let $q=\frac{4(m+1)+\mu}{2(m+1)+\frac{\mu}{2}-1}$, $a=1$, then $(p-1)a=q-2$, so Lemma \ref{katolemma} shows $F(t)$ blows up in finite time.

In the case of $n=1$, since $p=p_F(m+\frac{\mu}{2})$,
\begin{equation}
-\frac{\mu}{2}(p-1)-\frac{mp}{2}-\frac{(m+1)(n-1)(p-2)}{2}+2=2-\frac{m}{2}-\frac{2m+2\mu}{2m+\mu}\leq1
\end{equation}
holds, so $F(t)\geq C(t+M)$ by (\ref{primeprime}), substituting it into (\ref{prime2F}), we have $F^{\prime\prime}(t)\geq C (t+M)^{-2-\frac{4}{2m+\mu}}=C(t+M)^{-1},$ therefore,
\begin{equation}\label{kato2}
  \left\{
\begin{aligned}
&F^{\prime\prime}(t)\geq C (t+M)^{-2-\frac{4}{2m+\mu}}\vert F(t)\vert^p,\ t\geq T_0,\\
&F(t)\geq C (t+M)\log(t+M),
\end{aligned}
 \right.
  \end{equation}
let $q=2+\frac{4}{2m+\mu}$, $a=1$, then $(p-1)a=q-2$, so  $F(t)$ blows up in finite time.

\vskip 0.2cm
\noindent {\bf Acknowledgments}

Y-Q. Li is supported by the Postgraduate Research $\&$ Practice Innovation Program of Jiangsu Province (KYCX24$\_$1786). F. Guo is supported by the Priority Academic Program Development of Jiangsu Higher Education Institutions and the NSF of Jiangsu Province (BK20221320).

\renewcommand{\theequation}{A\arabic{equation}}
\setcounter{equation}{0}
%
%

\end{document}